\documentclass[10pt,english]{article}
\usepackage[utf8]{inputenc}
\usepackage[a4paper]{geometry}
\usepackage{babel}
\usepackage{mathtools}
\usepackage{dsfont}
\usepackage{amsmath}
\usepackage{amsthm}
\usepackage{amssymb}
\usepackage{stmaryrd}
\usepackage[unicode=true,
 bookmarks=true,bookmarksnumbered=false,bookmarksopen=false,
 breaklinks=false,pdfborder={0 0 1},backref=false,colorlinks=false]
 {hyperref}
\hypersetup{
 linkcolor=blue}

\makeatletter
\numberwithin{equation}{section}
\numberwithin{figure}{section}
\theoremstyle{plain}
\newtheorem{thm}{\protect\theoremname}[section]
\theoremstyle{remark}
\newtheorem{rem}[thm]{\protect\remarkname}
\theoremstyle{plain}
\newtheorem{prop}[thm]{\protect\propositionname}
\theoremstyle{plain}
\newtheorem{cor}[thm]{\protect\corollaryname}
\theoremstyle{plain}
\newtheorem{lem}[thm]{\protect\lemmaname}
\theoremstyle{definition}
\newtheorem{defn}[thm]{\protect\definitionname}

\usepackage{ dsfont }

\newcommand{\df}{\mathrm{d}}

\allowdisplaybreaks

\makeatother

\providecommand{\corollaryname}{Corollary}
\providecommand{\definitionname}{Definition}
\providecommand{\lemmaname}{Lemma}
\providecommand{\propositionname}{Proposition}
\providecommand{\remarkname}{Remark}
\providecommand{\theoremname}{Theorem}

\begin{document}
\global\long\def\df{\mathrm{def}}%
\global\long\def\eqdf{\stackrel{\df}{=}}%
\global\long\def\ep{\varepsilon}%
\global\long\def\ind{\mathds{1}}%

\title{Short geodesics and small eigenvalues on random hyperbolic punctured
spheres}
\author{Will Hide and Joe Thomas}

\maketitle

\begin{abstract}
We study the number of short geodesics and small eigenvalues on Weil-Petersson
random genus zero hyperbolic surfaces with $n$ cusps in the regime
$n\to\infty$. Inspired by work of Mirzakhani and Petri \cite{Mi.Pe19},
we show that the random multi-set of lengths of closed geodesics converges,
after a suitable rescaling, to a Poisson point process with explicit
intensity. As a consequence, we show that the Weil-Petersson probability
that a hyperbolic punctured sphere with $n$ cusps has at least $k=o(n)$
arbitrarily small eigenvalues tends to $1$ as $n\to\infty$.
\end{abstract}
{\footnotesize{}\tableofcontents{}}{\footnotesize\par}

\section{Introduction}

\subsection{Overview of main results}

For hyperbolic surfaces, understanding the lengths of closed geodesics
offers a deep insight into both the geometry and spectral theory of
the surface. In this paper, we consider the genus zero setting and
study the distribution of short closed geodesics on random surfaces
sampled from the moduli space $\mathcal{M}_{0,n}$ of hyperbolic punctured
spheres with respect to the Weil-Petersson probability measure $\mathbb{P}_{n}$,
as the number of cusps tends to infinity (see section \ref{sec:background}
for further details on this model). In particular, we show that the
number of short closed geodesics exhibit Poissonian statistics. Using
similar ideas, we gain an understanding about the number of small
Laplacian eigenvalues that exist on typical such surfaces.

To state these results more precisely, we introduce the following
notation. Let $X\in\mathcal{M}_{0,n}$ and let $0\leq a<b$ be real
numbers. Denote by $N_{n,[a,b]}(X):\left(\mathcal{M}_{0,n},\mathbb{P}_{n}\right)\to\mathbb{N}$
the random variable that counts the number of primitive closed geodesics
on $X$ with lengths in the interval $\left[\frac{a}{\sqrt{n}},\frac{b}{\sqrt{n}}\right]$.
We remark that when $n$ is sufficiently large, $\frac{b}{\sqrt{n}}<2\mathrm{arcsinh}(1)$,
and so in this case, the geodesics that are counted by the random
variable will be simple by the Collar Theorem \cite[Theorem 4.4.6]{Bu2010}.
In section \ref{sec:Poisson-Statistics} we prove the following, inspired
by the result of Mirzakhani and Petri \cite{Mi.Pe19} for closed surfaces
of large genus. Recall that on a probability space $(X,\mathbb{P})$,
a random variable $X:\Omega\to\mathbb{N}$ is called Poisson distributed
with mean $\lambda\in[0,\infty)$ if $\mathbb{P}(X=k)=\frac{\lambda^{k}e^{-\lambda}}{k!}$
for all $k\in\mathbb{N}$.
\begin{thm}
\label{thm:main-thm}Let $\ell\in\mathbb{N}$ and suppose that $0\leq a_{i}<b_{i}$
are real numbers for $i=1,\ldots,\ell$ such that the intervals $[a_{i},b_{i}]$
are pairwise disjoint. Then, the sequence of random vectors 
\[
\left(N_{n,[a_{1},b_{1}]}(X),\ldots,N_{n,[a_{\ell},b_{\ell}]}(X)\right)_{n\geq3}
\]
 converges in distribution as $n\to\infty$ to a vector of independent
Poisson distributed random variables with means
\[
\lambda_{[a_{i},b_{i}]}=\frac{b_{i}^{2}-a_{i}^{2}}{2}\frac{\left(j_{0}\pi\right)^{2}}{4}\left(1-\frac{J_{3}(j_{0})}{J_{1}(j_{0})}\right),
\]
where $J_{\alpha}$ are Bessel functions of the first kind and $j_{0}$
is the first positive zero of $J_{0}$.

\end{thm}

\begin{rem}
One can compute that $\frac{\left(j_{0}\pi\right)^{2}}{4}\left(1-\frac{J_{3}(j_{0})}{J_{1}(j_{0})}\right)\approx8.7997$.
It is equal to the constant $\sum_{i=2}^{\infty}\frac{V_{0,i+1}}{i!}\left(\frac{x_{0}}{2\pi^{2}}\right)^{i-1}$
where $x_{0}=-\frac{1}{2}j_{0}J_{0}'(j_{0})$ and $V_{0,m}$ is the
volume of the moduli space $\mathcal{M}_{0,m}$, and this constant
arises from using volume asymptotics of Manin and Zograf \cite{Ma.Zo00}
which we reproduce in Theorem \ref{thm:ZografManin}. 
\end{rem}

\begin{rem}
For $X\in\mathcal{M}_{0,n}$, let $\xi_{X}^{(n)}$ be the point measure
associated with the multiset $\{\sqrt{n}\ell_{\gamma}(X)\}_{\gamma}$
where the indexing runs over homotopy classes of closed curves on
an $n$ punctured sphere, and let $\xi^{(n)}$ denote the corresponding
point process on $[0,\infty)$ with respect to $(\mathcal{M}_{0,n},\mathbb{P}_{n})$.
Theorem \ref{thm:main-thm} is equivalent to saying that $\xi^{(n)}$
converges in distribution to a Poisson point process on $[0,\infty)$
with intensity $\frac{\left(j_{0}\pi\right)^{2}}{4}\left(1-\frac{J_{3}(j_{0})}{J_{1}(j_{0})}\right)x\mathrm{d}x$.
\end{rem}

This result is proven using the method of factorial moments (see Proposition
\ref{prop:factorial-moments}) which requires the computation of the
expectation of the random variables $N_{n,[a,b]}(X)(N_{n,[a,b]}(X)-1)\cdots(N_{n,[a,b]}(X)-k)$
for each $k=0,1,\ldots$. These random variables have the useful interpretation
of being the number of ordered lists of length $k+1$ consisting of
distinct primitive closed geodesics with lengths in $\left[\frac{a}{\sqrt{n}},\frac{b}{\sqrt{n}}\right]$.
The expected number of such lists is then ripe for computation using
the integration machinery developed by Mirzakhani \cite{Mi2007},
which we recall in Section \ref{sec:background}. A key insight of
our proof is classifying the topological types of multicurves that
make the dominant contribution to these expectations (see Definition
\ref{def:multicurve-types}).

An interesting consequence of Theorem \ref{thm:main-thm} is that
we can gain an understanding of the systole of a typical surface.
For example, we see that for $x>0$, 

\[
\lim_{n\to\infty}\mathbb{P}_{n}\left(\mathrm{sys}(X)<\frac{x}{\sqrt{n}}\right)=1-\exp\left(-\frac{x^{2}}{2}\frac{\left(j_{0}\pi\right)^{2}}{4}\left(1-\frac{J_{3}(j_{0})}{J_{1}(j_{0})}\right)\right),
\]
where $\mathrm{sys}(X)$ is the random variable measuring the length
of the systole of the surface $X$. We can compare this to the deterministic
bounds obtained by Schmutz \cite[Theorem 14]{Sc1994} which show that
for any $X\in\mathcal{M}_{0,n}$,

\begin{equation}
\mathrm{sys}(X)\leq4\mathrm{arcosh}\left(\frac{3n-6}{n}\right),\label{eq:systole-bd}
\end{equation}
and this bound is sharp for $n=4,6$ and $12$. See also Lakeland
and Young \cite{La.Yo22} for sharper bounds on the systole in the
case of arithmetic hyperbolic punctured spheres. We show that asymptotically
almost surely (that is, with probability tending to 1 as $n\to\infty$)
the systole of a surface is much smaller than this in the Weil-Petersson
model. In fact, we also have the complementary result.
\begin{prop}
\label{prop:prop1.3}There exists a constant $B>0$ such that for
any constants $0<\varepsilon<\frac{1}{2}$, $A>0$, any $c_{n}<An^{\varepsilon}$
and $n$ sufficiently large, 
\[
\mathbb{P}_{n}\left(\mathrm{sys}\left(X\right)>\frac{c_{n}}{\sqrt{n}}\right)\leq\max\left\{ Bc_{n}^{-2},\frac{B}{\sqrt{n}}\right\} .
\]
\end{prop}

\begin{rem}
Proposition \ref{prop:prop1.3} also remains true on $\mathcal{M}_{g,n}$
when the genus $g$ is fixed and non-zero.

\end{rem}

\begin{rem}
\label{thm:thm1.4}By similar methods to \cite[Theorem 5.1]{Mi.Pe19}
it is possible to write

\[
\sqrt{n}\mathbb{E}_{n}\left(\mathrm{sys}(X)\right)=\int_{0}^{\infty}\mathbb{P}_{n}\left(\mathrm{sys}\left(X\right)>\frac{x}{\sqrt{n}}\right)\mathrm{d}x=\int_{0}^{\infty}\mathbb{P}_{n}\left(N_{n,[0,x]}(X)=0\right)\mathrm{d}x.
\]
By taking the $n\to\infty$ limit, interchanging the limit and integral
and using the convergence in distribution of $N_{n,[0,x]}(X)$ from
Theorem \ref{thm:main-thm}, one would obtain 
\begin{equation}
\lim_{n\to\infty}\sqrt{n}\mathbb{E}_{n}\left(\mathrm{sys}\left(X\right)\right)=\frac{\sqrt{2}}{j_{0}\sqrt{\pi}\sqrt{1-\frac{J_{3}(j_{0})}{J_{1}(j_{0})}}}\approx0.4225.\label{eq:expectation}
\end{equation}
However, to justify this interchange of limit and integral we require
better bounds on $\mathbb{P}_{n}\left(\mathrm{sys}\left(X\right)>\frac{x}{\sqrt{n}}\right)$
than those obtained in Proposition \ref{prop:prop1.3} to apply the
dominated convergence theorem. By (\ref{eq:systole-bd}) it would
be sufficient to obtain an improvement of the bound in Proposition
(\ref{prop:prop1.3}) to just $Bc_{n}^{-2}$. The extra term $\frac{B}{\sqrt{n}}$
in Proposition \ref{prop:prop1.3} essentially arises from using the
trivial bounds in Lemma \ref{lem:volbds}. In a forthcoming work of
the authors, we prove a large-$n$ asymptotic for $V_{g,n}\left(x_{1},\dots,x_{k}\right)$
from which (\ref{eq:expectation}) can be deduced. However the methods
of the current paper fall short of this so we do not claim it here.
\end{rem}

Our methods also allow us to consider vectors consisting of random
variables counting different topological types of short geodesics.
For an integer $c\geqslant2$ and real numbers $a,b\geqslant0$, we
let $N_{n,c,[a,b]}(X):\left(\mathcal{M}_{0,n},\mathbb{P}_{n}\right)\to\mathbb{N}$
denote the random variable which counts the number of primitive closed
geodesics with lengths in $\left[\frac{a}{\sqrt{n}},\frac{b}{\sqrt{n}}\right]$
which bound $c$ cusps. We prove the following.
\begin{thm}
\label{thm:main-thm-2}Let $\ell\in\mathbb{N}$, $c_{1},\ldots,c_{\ell}\geq2$
be distinct integers and $0\leq a_{i}<b_{i}$ be real numbers for
$i=1,\ldots,\ell$. Then, the sequence of random vectors 
\[
\left(N_{n,c_{1},[a_{1},b_{1}]}(X),\ldots,N_{n,c_{\ell},[a_{\ell},b_{\ell}]}(X)\right){}_{n\geq3}
\]
 converges in distribution as $n\to\infty$ to a vector of independent
Poisson distributed random variables with means 

\[
\lambda_{c_{i},[a_{i},b_{i}]}=\frac{b_{i}^{2}-a_{i}^{2}}{2}\frac{V_{0,c_{i}+1}}{c_{i}!}\left(\frac{x_{0}}{2\pi^{2}}\right)^{c_{i}-1},
\]
where $V_{0,m}$ is the volume of the moduli space $\mathcal{M}_{0,m}$
and $x_{0}=-\frac{1}{2}j_{0}J_{0}'(j_{0})$ with $J_{0}$ the Bessel
function of the first kind and $j_{0}$ its first positive zero.
\end{thm}

As a consequence of Theorem \ref{thm:main-thm-2}, we gain an understanding
of the topological nature of the systole in the large $n$ limit.
This can be seen from the following corollary.
\begin{cor}
\label{cor:systole-type}Let $k\geq2$ be an integer, $n\in\mathbb{N}$
and $x>0$ be real. Suppose that $\mathcal{A}_{k,x,n}\subseteq\mathcal{M}_{0,n}$
is the collection of surfaces $X$ that satisfy the following conditions:
\begin{enumerate}
\item for each $2\leq i<k$, no systolic curve of X separates off exactly
$i$ cusps, 
\item there exists a systolic curve on $X$ that separates off at least
$k$ cusps, with length less than $\frac{x}{\sqrt{n}}$.
\end{enumerate}
Then, 

\[
\lim_{n\to\infty}\mathbb{\mathbb{P}}_{n}\left(\mathcal{A}_{k,x,n}\right)\geq e^{-\frac{x^{2}}{2}\sum_{i=2}^{k-1}\frac{V_{0,i+1}}{i!}\left(\frac{x_{0}}{2\pi^{2}}\right)^{i-1}}\left(1-e^{-\frac{x^{2}}{2}\frac{V_{0,k+1}}{k!}\left(\frac{x_{0}}{2\pi^{2}}\right)^{k-1}}\right)>0.
\]

\end{cor}

\begin{rem}
It is also not too difficult to show that 2. in Corollary \ref{cor:systole-type}
can be changed to \textbf{exactly }$k$ cusps. Indeed, this follows
from an analogue of Theorem \ref{thm:main-thm-2} for the random vector
\[
\left(N_{n,c_{1},[a_{1},b_{1}]}(X),\ldots,N_{n,c_{\ell},[a_{\ell},b_{\ell}]}(X),N_{n,\geq\max(c_{i})+1,[a_{\ell+1},b_{\ell+1}]}(X)\right){}_{n\geq3},
\]
where $N_{n,\geq c,[a,b]}(X)$ counts the number of primitive geodesics
with lengths in $\left[\frac{a}{\sqrt{n}},\frac{b}{\sqrt{n}}\right]$
that separate off at least $c$ cusps. 
\end{rem}

Understanding the distribution of closed geodesics on a surface also
offers insight into its spectral properties. In Section \ref{sec:Spectrum}
we demonstrate the existence of many short closed geodesics on typical
hyperbolic punctured spheres that each separate off two distinct cusps
from the surface. The existence of these curves, when combined with
the Mini-max Lemma \ref{lem:Minimax} for the Laplacian, and an argument
similar to Buser \cite[Theorem 8.1.3]{Bu2010}, can be used to deduce
the existence of $o\left(n\right)$ arbitrarily small eigenvalues
on a typical hyperbolic punctured sphere. Recall, that the spectrum
of a hyperbolic punctured sphere consists of absolutely continuous
spectrum in the range $\left[\frac{1}{4},\infty\right)$, a simple
eigenvalue at $0$, possibly finitely many eigenvalues in the range
$\left(0,\frac{1}{4}\right)$ and potentially embedded eigenvalues
above $\frac{1}{4}$. Then, let $\lambda_{k}\left(X\right)$ denote
the $(k+1)^{\mathrm{th}}$ smallest eigenvalue of the Laplacian on
$X\in\mathcal{M}_{0,n}$ $\textit{if it exists}$. A Theorem of Zograf
\cite{Zo1987} says that there is a constant $C>0$ such that for
any $X\in\mathcal{M}_{g,n}$, $\lambda_{1}\left(X\right)\leqslant C\frac{g+1}{n}$.
 In particular, if $g=o(n)$ then any surface in $\mathcal{M}_{g,n}$
has a small eigenvalue. Our next result complements this by showing
a random surface has many small eigenvalues. We prove the following.
\begin{thm}
\label{thm:small-eigenvalues}There is a constant $C>0$ such that
for any function $k:\mathbb{N}\to\mathbb{N}$ with $k=o\left(n\right)$
and $k\to\infty$ as $n\to\infty$, then 

\[
\mathbb{P}_{n}\left(\lambda_{k}(X)<C\sqrt{\frac{k}{n}}\right)\to1,
\]
as $n\to\infty$. In particular, for any $\ep>0$, $\mathbb{P}_{n}\left[\lambda_{k}\left(X\right)<\ep\right]\to1$
as $n\to\infty$.
\end{thm}

By work of Ballmann, Mathiesen and Mondal \cite{Ba.Ma.Mo17} we have
that $\lambda_{n-1}(X)>\frac{1}{4}$, in particular Theorem \ref{thm:small-eigenvalues}
says that a random surface with many cusps is not far from saturating
this bound to leading order.
\begin{rem}
Theorem \ref{thm:small-eigenvalues} can be easily extended to surfaces
with fixed genus $g>0$, c.f. Remark \ref{rem:g>0}.
\end{rem}

\subsection{Relations to existing work}

It is worthwhile to compare the results obtained here with existing
literature in the large genus and mixed large genus and large cusp
regimes. For closed hyperbolic surfaces of genus $g$, Mirzakhani
and Petri \cite{Mi.Pe19} have also obtained Poissonian statistics
for the distributions of closed geodesics in the regime $g\to\infty$,
and their work is the main inspiration for our investigation here.
More precisely, they consider the random variables $N_{g,[a,b]}(X):(\mathcal{M}_{g},\mathbb{P}_{g})\to\mathbb{N}$,
where $\mathbb{P}_{g}$ is the associated Weil-Petersson probability
measure, which count the number of primitive closed geodesics on the
surface $X$ with lengths in the interval $[a,b]$. In the $g\to\infty$
limit, they show that these random variables converge in distribution
to a Poisson distributed random variable with mean

\[
\lambda_{[a,b]}^{'}=\int_{a}^{b}\frac{e^{t}+e^{-t}-2}{2t}\mathrm{d}t.
\]

\begin{rem}
We note that $\frac{e^{t}+e^{-t}-2}{2t}=\frac{2\sinh^{2}(\frac{t}{2})}{t}=\frac{t}{2}+O(t^{3}),$
and so for small $a$ and $b$ the integral gives a leading order
of $\frac{b^{2}-a^{2}}{4}$ akin to the constant arising in Theorem
\ref{thm:main-thm}. 

They use this to show that the limit of the expected systole length
is $\approx1.61498...$. This is in contrast to the hyperbolic punctured
sphere setting that we study here, where the systole is on scales
of order $n^{-\frac{1}{2}}$. In fact, in the large genus limit, Mirzakhani
\cite[Theorem 4.2]{Mi2013} showed that for $\varepsilon>0$ sufficiently
small, the Weil-Petersson probability of a closed hyperbolic surface
having systole smaller than $\varepsilon$ is proportional to $\varepsilon^{2}$.
Moreover, if one considers the length of the separating systole, that
is, the shortest closed geodesic that separates the surface, its expected
size is $2\log(g)$ in the large genus regime by the work of Parlier,
Wu and Xue \cite{Pa.Wu.Xu21}. Note, in the setting we consider here,
the systole is always separating. 

In addition to this, Nie, Wu and Xue showed in \cite{Ni.Wu.Xu20}
that with probability tending to $1$ as $g\to\infty$, the separating
systole on $X\in\mathcal{M}_{g}$ separates $X$ into $\Sigma_{1,1}\cup\Sigma_{g-1,1}$.
In contrast, Corollary \ref{cor:systole-type} demonstrates that there
is not an analogue of this result for the hyperbolic punctured sphere
setting. Indeed, with positive probability as $n\to\infty$, a hyperbolic
surface has only separating systoles that cut off at least $k$ cusps
for any $k\geq2$.
\end{rem}

For the spectral consequences, there is a stark contrast between the
cusp and genus limits. Indeed, in the closed hyperbolic surface setting,
it is conjectured that typical surfaces (with respect to a reasonable
probability model such as the Weil-Petersson model) should have no
small, non-trivial eigenvalues. More precisely, Wright makes the conjecture
\cite[Problem 10.4]{Wr2020} that for any $\varepsilon>0$,

\begin{equation}
\lim_{g\to\infty}\mathbb{P}_{g}\left(X:\lambda_{1}(X)\geq\frac{1}{4}-\varepsilon\right)=1,\label{eq:max-spec-gap}
\end{equation}
where $\lambda_{1}(X)$ is the first non-zero eigenvalue of the Laplacian
on $X$. The current state of the art for this is $\frac{1}{4}$ replaced
by $\frac{2}{9}$ in (\ref{eq:max-spec-gap}) above obtained by Anantharaman
and Monk\cite{An.Mo2023} . See also the work of Wu and Xue \cite{Wu.Xu2022}
and Lipnowski and Wright \cite{Li.Wr21} where $\frac{1}{4}$ is replaced
by $\frac{3}{16}$. These results proceed the earlier work of Magee,
Naud and Puder \cite{Ma.Na.Pu2022} who obtain a relative spectral
gap of size $\frac{3}{16}$ in a random covering probability model.
More precisely, given a compact hyperbolic surface, one can consider
a degree $d$ Riemannian covering uniformly at random. Their result
then states that for any $\varepsilon>0$, a covering will have no
new eigenvalues from those on the base surface in the interval $[0,\frac{3}{16}-\varepsilon]$
with probability tending to 1 as $d\to\infty$.

In the non-compact setting, Magee and the first named author \cite{Hi.Ma22}
recently showed that in the same random covering model, this relative
spectral gap can be improved to the interval $[0,\frac{1}{4}-\varepsilon]$.
In particular, via a compactification procedure, they demonstrated
the existence of a sequence of closed hyperbolic surfaces with genus
tending to $\infty$ with $\lambda_{1}\to\frac{1}{4}$.

In the mixed regime of genus $g\to\infty$ and number of cusps $n=O(g^{\alpha})$
for $0\leqslant\alpha<\frac{1}{2}$, the first named author \cite{Hi22}
demonstrated in the Weil-Petersson model that with probability tending
to $1$ as $g\to\infty$, a surface has an explicit spectral gap with
size dependent upon $\alpha$. Moving past the $n=o(g^{\frac{1}{2}})$
threshold, Shen and Wu \cite{Sh.Wu22} showed that for $g^{\frac{1}{2}+\delta}\ll n\ll g$
and any $\varepsilon>0$, the Weil-Petersson probability of $\lambda_{1}$
being less than $\varepsilon$ tends to one as $g\to\infty$. For
the scale $n\asymp g^{\frac{1}{2}}$ as $g\to\infty$, Shen and Wu
\cite[Theorem 2]{Sh.Wu22} determine that for any $\varepsilon>0$,
$\lambda_{1}<\varepsilon$ with positive probability as $g\to\infty$.
Theorem \ref{thm:small-eigenvalues} complements these results by
providing information about $o\left(n\right)$ eigenvalues. 

Other related work in the Weil-Petersson model includes the study
of Laplacian eigenfunctions \cite{Gi.Le.Sa.Th21,Th2020}, quantum
ergodicity \cite{Le.Sa17,Le.Sa20}, local Weyl law \cite{Mo2020},
Gaussian Orthogonal Ensemble energy statistics \cite{Ru22} (see \cite{Na2022}
in the case of random covers), and prime geodesic theorem error estimates
\cite{Wu.Xu22a}. See also \cite{Gu.Pa.Yo11} and the references therein
for results concerning the lengths of boundaries of pants decompositions
for random surfaces.

\section{Background\label{sec:background}}

In this section we introduce the necessary background on moduli space,
the Weil-Petersson metric and Mirzakhani's integration formula. One
can see \cite{Wr2020} for a nice exposition of these topics.

\subsection{Moduli space }

Let $\Sigma_{g,c,d}$ denote a topological surface with genus $g$,
$c$ labeled punctures and $d$ labeled boundary components where
$2g+n+d\geqslant3$. A marked surface of signature $\left(g,c,d\right)$
is a pair $\left(X,\varphi\right)$ where $X$ is a hyperbolic surface
and $\varphi:\Sigma_{g,c,d}\to X$ is a homeomorphism. Given $\left(l_{1},...,l_{d}\right)\in\mathbb{R}_{>0}^{d}$,
we define the Teichmüller space $\mathcal{T}_{g,c+d}\left(l_{1},\dots,l_{d}\right)$
by
\begin{align*}
\mathcal{T}_{g,c,d}\left(l_{1},...,l_{d}\right) & \stackrel{\text{def}}{=}\left\{ \substack{\text{\text{Marked surfaces} }\left(X,\varphi\right)\text{ of signature }\left(g,c,d\right)\\
\text{with labeled totally geodesic boundary components }\\
\left(\beta_{1},\dots,\beta_{d}\right)\text{ with lengths }\left(l_{1},\dots,l_{d}\right)
}
\right\} /\sim,
\end{align*}
where $\left(X_{1},\varphi_{1}\right)\sim\left(X_{2},\varphi_{2}\right)$
if and only if there exists an isometry $m:X_{1}\to X_{2}$ such that
$\varphi_{2}$ and $m\circ\varphi_{1}$ are isotopic. Let $\text{Homeo}^{+}\left(\Sigma_{g,c,d}\right)$
denote the group of orientation preserving homeomorphisms of $\Sigma_{g,c,d}$
which leave every boundary component setwise fixed and do not permute
the punctures. Let $\text{Homeo}_{0}^{+}\left(\Sigma_{g,c,d}\right)$
denote the subgroup of homeomorphisms isotopic to the identity. The
mapping class group is defined as 
\[
\text{MCG}_{g,c,d}\stackrel{\text{def}}{=}\text{Homeo}^{+}\left(\Sigma_{g,c,d}\right)/\text{Homeo}_{0}^{+}\left(\Sigma_{g,c,d}\right).
\]
$\text{Homeo}^{+}\left(\Sigma_{g,c,d}\right)$ acts on $\mathcal{T}_{g,c,d}\left(l_{1},...,l_{d}\right)$
by pre-composition of the marking, and $\text{Homeo}_{0}^{+}\left(\Sigma_{g,c,d}\right)$
acts trivially, hence $\text{MCG}_{g,c,d}$ acts on $\mathcal{T}_{g,c,d}\left(l_{1},...,l_{d}\right)$
and we define the moduli space $\mathcal{M}_{g,c,d}\left(l_{1},...,l_{d}\right)$
by 
\[
\mathcal{M}_{g,c,d}\left(l_{1},...,l_{d}\right)\stackrel{\text{def}}{=}\mathcal{T}_{g,c,d}\left(l_{1},...,l_{d}\right)/\text{MCG}_{g,c,d}.
\]
By convention, a geodesic of length $0$ is a cusp and we suppress
the distinction between punctures and boundary components in our notation
by allowing $l_{i}\geqslant0$. In particular,

\[
\mathcal{M}_{g,c+d}=\mathcal{M}_{g,c,d}\left(0,\dots,0\right).
\]
Throughout the sequel we shall restrict our study to the case that
$g=0$.

\subsection{Weil-Petersson metric}

For $\boldsymbol{l}=\left(l_{1},\dots,l_{n}\right)$ with $l_{i}\geqslant0$
for $1\leqslant i\leqslant n$, the space $\mathcal{T}_{g,n}\left(\boldsymbol{l}\right)$
carries a natural symplectic structure known as the Weil-Petersson
symplectic form and is denoted by $\omega_{WP}$ \cite{Go1984}. It
is invariant under the action of the mapping class group and descends
to a symplectic form on $\mathcal{M}_{g,n}\left(\boldsymbol{l}\right)$.
The form $\omega_{WP}$ induces the volume form
\[
\text{dVol}_{WP}\stackrel{\text{def}}{=}\frac{1}{\left(3g-3+n\right)!}\bigwedge_{i=1}^{3g-3+n}\omega_{WP},
\]
which is also invariant under the action of the mapping class group
and descends to a volume form on $\mathcal{M}_{g,n}\left(\boldsymbol{l}\right)$.
By a theorem of Wolpert \cite{Wo1981}, this volume form can be made
explicit in terms of Fenchel-Nielsen coordinates. We write $\mathrm{d}X$
as shorthand for $\text{dVol}_{WP}$. We let $V_{g,n}\left(\boldsymbol{l}\right)$
denote $\text{Vol}_{WP}\left(\mathcal{M}_{g,n}\left(\boldsymbol{l}\right)\right)$,
the total volume of $\mathcal{M}_{g,n}\left(\boldsymbol{l}\right)$,
which is finite. We write $V_{g,n}$ to denote $V_{g,n}\left(\boldsymbol{0}\right)$
and since we shall only consider the case $g=0$ in this article,
we shall often write $V_{n}\eqdf V_{0,n}$.

We define the Weil-Petersson probability measure on $\mathcal{M}_{0,n}$
by normalizing $\text{dVol}_{WP}$. Indeed, for any Borel subset $\mathcal{B\subseteq\mathcal{M}}_{0,n}$,
\[
\mathbb{P}_{n}\left[\mathcal{B}\right]\stackrel{\text{def}}{=}\frac{1}{V_{n}}\int_{\mathcal{M}_{0,n}}\ind_{\mathcal{B}}\mathrm{d}X.
\]
 We write $\mathbb{E}_{n}$ to denote expectation with respect to
$\mathbb{P}_{n}$.

\subsection{Mirzakhani's integration formula}

We define a $k$-multicurve to be an ordered $k$-tuple $\Gamma=\left(\gamma_{1},...,\gamma_{k}\right)$
of disjoint non-homotopic non-peripheral simple closed curves on $\Sigma_{0,n}$
and we write $\left[\Gamma\right]$$=\left[\gamma_{1},...,\gamma_{k}\right]$
to denote its homotopy class. The mapping class group $\text{MCG}_{0,n}$
acts on homotopy classes of multicurves and we denote the orbit containing
$\left[\Gamma\right]$ by 
\[
\mathcal{O}_{\Gamma}=\left\{ \left(g\cdot\gamma_{1},...,g\cdot\gamma_{k}\right)\mid g\in\text{MCG}_{0,n}\right\} .
\]
Given a simple, non-peripheral closed curve $\gamma$ on $\Sigma_{0,n}$,
for $\left(X,\varphi\right)\in\mathcal{T}_{0,n}$ we define $\ell_{\gamma}\left(X\right)$
to be the length of the unique geodesic in the free homotopy class
of $\varphi\left(\gamma\right)$. Then given a function $f:\mathbb{R}_{\geqslant0}^{k}\to\mathbb{R}_{\geqslant0}$,
for $X\in\mathcal{M}_{0,n}$ we define 
\[
f^{\Gamma}\left(X\right)\eqdf\sum_{\left(\alpha_{1},...,\alpha_{k}\right)\in\mathcal{O}_{\Gamma}}f\left(\ell_{\alpha_{1}}\left(X\right),...,\ell_{\alpha_{k}}\left(X\right)\right),
\]
which is well defined on $\mathcal{M}_{0,n}$ since we sum over the
orbit $\mathcal{O}_{\Gamma}$. Let $\Sigma_{n}\left(\Gamma\right)$
denote the result of cutting the surface $\Sigma_{0,n}$ along $\left(\gamma_{1},...,\gamma_{k}\right)$,
then $\Sigma_{n}\left(\Gamma\right)=\sqcup_{i=1}^{s}\Sigma_{0,c_{i},d_{i}}$
for some $\left\{ \left(c_{i},d_{i}\right)\right\} _{i=1}^{s}$. Each
$\gamma_{i}$ gives rise to two boundary components $\gamma_{i}^{1}$
and $\gamma_{i}^{2}$ of $\Sigma_{n}\left(\Gamma\right)$. Given $\boldsymbol{x}=\left(x_{1},...,x_{k}\right)$,
let $\boldsymbol{x}^{(i)}$ denote the tuple of coordinates $x_{j}$
of $\boldsymbol{x}$ such that $\gamma_{j}$ is a boundary component
of $\Sigma_{c_{i},d_{i}}$. We define 
\begin{align*}
V_{n}\left(\Gamma,\underline{x}\right)\stackrel{\text{def}}{=} & \prod_{i=1}^{s}V_{c_{i}+d_{i}}\left(\boldsymbol{x}^{(i)}\right).
\end{align*}
We can now state Mirzakhani's integration formula.
\begin{thm}[{Mirzakhani's Integration Formula \cite[Theorem 7.1]{Mi2007}}]
\label{thm:MIF}Given a $k$-multicurve $\Gamma=\left(\gamma_{1},\dots,\gamma_{k}\right)$,
\[
\int_{\mathcal{M}_{0,n}}f^{\Gamma}\left(X\right)dX=\int_{\mathbb{R}_{\geqslant0}^{k}}f\left(x_{1},...,x_{k}\right)V_{n}\left(\Gamma,\underline{x}\right)x_{1}\cdots x_{k}\mathrm{d}x_{1}\cdots\mathrm{d}x_{k}.
\]
\end{thm}

\subsection{Volume bounds}

Finally we state some results on Weil-Petersson volumes of moduli
space which we will need later on. We make essential use of the following
asymptotic for $V_{0,n}$ due to Manin and Zograf \cite[Theorem 6.1]{Ma.Zo00}.

\begin{thm}
\label{thm:ZografManin}There exists a constant $B_{0}$ such that
as $n\to\infty$, 
\begin{align*}
V_{n} & =(2\pi^{2})^{n-3}n!(n+1)^{-\frac{7}{2}}x_{0}^{-n}\left(B_{0}+O\left(\frac{1}{n}\right)\right),
\end{align*}
where $x_{0}=-\frac{1}{2}j_{0}J_{0}'(j_{0})$ and $J_{0}$ is the
Bessel function and $j_{0}$ is the first positive zero of $J_{0}$.
\end{thm}

The following bound of Mirzakhani \cite[Lemma 3.2(3)]{Mi2013} is
needed for comparing moduli space volumes with differing cusps.
\begin{lem}
\label{lem:mirzakhani-vol-comparison}There exist constants $\gamma_{1},\gamma_{2}>0$
such that for any $n\geq3$,
\[
\gamma_{1}\frac{V_{n+1}}{n-2}\leq V_{n}\leq\gamma_{2}\frac{V_{n+1}}{n-2}.
\]
\end{lem}

The following well known bounds will be sufficient for our purposes,
e.g. \cite[Proposition 3.1]{Mi.Pe19}.
\begin{lem}
\label{lem:volbds}Suppose that $n\geq3$, then for any $b_{1},\ldots,b_{n}\geq0$,

\[
1\leqslant\frac{V_{0,n}\left(2b_{1},\dots,2b_{n}\right)}{V_{0,n}}\leqslant\prod_{i=1}^{n}\frac{\sinh\left(b_{i}\right)}{b_{i}}.
\]
\end{lem}

\section{Poisson Statistics\label{sec:Poisson-Statistics}}

In this section, we will prove Theorems \ref{thm:main-thm} and \ref{thm:main-thm-2}.
We begin by proving Theorem \ref{thm:main-thm} for $\ell=1$, that
is, for a sequence of random variables; the extension to random vectors
is straightforward, and we outline the necessary changes in subsection
\ref{subsec:main-thm-proof} so that we can drastically simplify indexing
and notation here. Furthermore, many of the same ideas can be used
for the proof of Theorem \ref{thm:main-thm-2} which we will prove
in subsection \ref{subsec:main-thm2-proof}.

To this end, for $0\leq a<b$ and $X\in\mathcal{M}_{0,n}$, recall
that we defined $N_{n,[a,b]}(X)$ to be the number of primitive closed
geodesics on $X$ whose length are in the interval $\left[\frac{a}{\sqrt{n}},\frac{b}{\sqrt{n}}\right]$.
For fixed $a,b$, we will show that in the $n\to\infty$ regime, the
random variables $N_{n,[a,b]}(X)$ converge in distribution to a Poisson
distributed random variable with mean 

\[
\lambda=\frac{b^{2}-a^{2}}{2}\cdot\sum_{i=2}^{\infty}\frac{V_{i+1}}{i!}x_{0}^{i-1}.
\]
We then manipulate this series to obtain the stated constant in terms
of Bessel functions. To demonstrate convergence, we use the method
of factorial moments.
\begin{prop}[Method of factorial moments]
\label{prop:factorial-moments} Let $\ell\in\mathbb{N}$ and suppose
that $(\Omega_{n},\mathbb{\mathbb{P}}_{n})_{n\geq1}$ are a sequence
of probability spaces, and $X_{i,n}:\Omega_{n}\to\mathbb{N}$ are
a sequence of random variables for $i=1,\ldots,\ell$. For $k\in\mathbb{N}$,
we denote by

\[
\left(X_{i,n}\right)_{k}:=X_{i,n}(X_{i,n}-1)\cdots(X_{i,n}-k+1).
\]
Suppose that there exists $\lambda_{i}\in(0,\infty)$ for $i=1,\ldots,\ell$
such that for every $k_{1},\ldots,k_{\ell}\in\mathbb{N}$,

\[
\lim_{n\to\infty}\mathbb{E}_{n}\left[\left(X_{1,n}\right)_{k_{1}}\cdots\left(X_{\ell,n}\right)_{k_{\ell}}\right]=\lambda_{1}^{k_{1}}\cdots\lambda_{\ell}^{k_{\ell}}.
\]
Then, the random vector $(X_{1,n},\ldots,X_{\ell,n})$ converges in
distribution to a random vector of independent Poisson distributed
random variable with parameter $\lambda_{i}$. 
\end{prop}

\noindent We will first apply Proposition \ref{prop:factorial-moments}
in the case that $\ell=1$. Since we are interested in understanding
the expectation of the random variables as $n\to\infty$ for fixed
$0\leq a<b$, we can assume that $n$ is sufficiently large so that
$\frac{b}{\sqrt{n}}<2\mathrm{arcsinh}(1)$. Then $\left(N_{n,[a,b]}(X)\right)_{k}$
is precisely the number of $k$-tuples consisting of distinct, disjoint,
primitive simple closed geodesics on $X$ whose lengths are in the
interval $\left[\frac{a}{\sqrt{n}},\frac{b}{\sqrt{n}}\right]$. Disjointness
of the curves follows from the fact that any two geodesics of length
less than $2\mathrm{arcsinh}(1)$ do not intersect \cite[Theorem 4.1.6]{Bu2010}.

\noindent 
\begin{rem}
\label{rem:mcg-orbits}Consider the mapping class group orbit of an
ordered multicurve. Since the mapping class group fixes the punctures
of the surface, any two disjoint, non-peripheral, simple closed curves
on $\Sigma_{0,n}$ will be in the same mapping class group orbit if
and only if they separate the same punctures from the surface. Moreover,
two ordered\textbf{ }multicurves are in the same mapping class group
orbit if and only if their corresponding curve components each separate
the surface with the same topological decomposition and the punctures
of $\Sigma_{0,n}$ that are on the subsurfaces are the same. 
\end{rem}

We will separate out the ordered multicurves that we consider into
two types: nested multicurves and unnested multicurves. 
\begin{defn}
\label{def:multicurve-types}A \textbf{nested multicurve} is an ordered
multicurve $\Gamma=(\gamma_{1},\ldots,\gamma_{k})$ on $\Sigma_{0,n}$
consisting of distinct, disjoint, non-peripheral simple closed curves
such that for some $i$, the two subsurfaces in the disconnected surface
$\Sigma_{0,n}\setminus\gamma_{i}$ each contain at least one of the
remaining multicurve components $\gamma_{j}$ in their interior. An
\textbf{unnested multicurve }is an ordered multicurve that is not
nested. In other words, it is an ordered multicurve $\Gamma=(\gamma_{1},\ldots,\gamma_{k})$
on $\Sigma_{0,n}$ consisting of distinct, disjoint, non-peripheral
simple closed curves such that for every $i$, one of the subsurfaces
in the disconnected surface $\Sigma_{0,n}\setminus\gamma_{i}$ contains
all of the other multicurve components in its interior.
\end{defn}

We will write 
\begin{equation}
\left(N_{n,[a,b]}(X)\right)_{k}=N_{n,[a,b]}^{N,k}(X)+N_{n,[a,b]}^{U,k}(X),\label{eq:nested + unnested}
\end{equation}
where $N_{n,[a,b]}^{N,k}(X)$ counts the nested $k$-multicurves and
$N_{n,[a,b]}^{U,k}(X)$ counts the unnested $k$-multicurves. We will
see that the main contribution to the expectation arises from the
unnested multicurves.

\subsection{Contribution of nested multicurves}

\label{subsec:nested}

Recall from Section \ref{sec:background} that $\Sigma_{n}$ denotes
a topological surface with $n$ punctures labeled with the alphabet
$\{1,\ldots,n\}$. More generally, $\Sigma_{g,c,d}$ denotes a topological
surface with genus $g$, $c$ labeled punctures and $d$ labeled boundaries.

Given a nested ordered $k$-multicurve $\Gamma$ on $\Sigma_{n}$,
we associate an unordered collection of $k+1$ triples $\left\{ \left(c_{i},d_{i},\left\{ a_{1}^{i},\ldots,a_{c_{i}}^{i}\right\} \right)\right\} _{i=1}^{k+1}$
where 
\begin{enumerate}
\item $\sum_{i=1}^{k+1}c_{i}=n$, $c_{i}\geq0,$
\item $\sum_{i=1}^{k+1}d_{i}=2k$, $1\leq d_{i}<k,$
\item $c_{i}+d_{i}\geq3,$
\item $\left\{ a_{1}^{i},\ldots,a_{c_{i}}^{i}\right\} \subseteq\{1,\ldots,n\}$
are pairwise disjoint and their union over all $i$ is $\{1,\ldots,n\}$,
\item Cutting $\Sigma_{n}$ along the multicurve $\Gamma$ gives rise to
\end{enumerate}
\[
\Sigma_{n}\setminus\Gamma=\bigsqcup_{i=1}^{k+1}\Sigma_{0,c_{i},d_{i},}
\]
where $\Sigma_{0,c_{i},d_{i}}$ has the labels $\left\{ a_{1}^{i},\ldots,a_{c_{i}}^{i}\right\} $
on its punctures inherited from the puncture labels on $\Sigma_{n}$.
Indeed, the decomposition is obtained by recording the puncture and
boundary numbers on the subsurfaces in the fifth condition along with
the corresponding labels of the punctures.

Under such a cutting for the resulting tuples, condition 1 is immediate
since any subsurface has at least zero punctures and the sum of the
number of punctures of each subsurface is precisely $n$ since the
subsurfaces glue back together to give $\Sigma_{n}$. 

Each component of $\Gamma$ gives rise to exactly two boundary components
in the decomposition so that the sum of the boundaries is $2k$. Moreover,
each subsurface has at least one boundary obtained from the multicurve
component that separates it from the surface. The number of boundary
components on each subsurface is strictly less than $k$ precisely
because the multicurve is nested. Indeed, if one of the subsurfaces
has $k$ boundaries, then by $\sum_{i=1}^{k+1}d_{i}=2k$, each of
the other $k$ subsurfaces has precisely one boundary corresponding
to the $k$ components in $\Gamma$. This means that cutting $\Sigma_{n}$
by a curve component of $\Gamma$ gives two subsurfaces, one of which
corresponds to a subsurface in the decomposition with one boundary
component. This subsurface hence does not contain any other multicurve
components and thus the multicurve is unnested. Put together, these
observations mean that condition 2 holds for the collection of tuples
arising from $\Gamma$. 

Condition 3 holds since otherwise in the cut surface there would be
a component with exactly one boundary and one puncture or with two
boundary components and no punctures. In the first case, the curve
component in the multicurve that corresponds to this boundary component
homotopes down to the puncture which is a contradiction because the
curves are non-peripheral. In the second case, two curve components
must bound an annulus which contradicts the assumption that the curve
components are non-homotopic.

Denote the collection of unordered triples satisfying the conditions
(1)-(4) by $\mathcal{\tilde{A}}_{k}$.

\begin{lem}
\label{lem:nested-curve-count}A collection $\left\{ \left(c_{i},d_{i},\left\{ a_{1}^{i},\ldots,a_{c_{i}}^{i}\right\} \right)\right\} _{i=1}^{k+1}\in\mathcal{\tilde{A}}_{k}$
corresponds to exactly $f(d_{1},\ldots,d_{k+1})$ mapping class group
orbits of nested multicurves for some function $f(d_{1},\ldots,d_{k+1})\leq(2k-1)!!k!$.
\end{lem}

\begin{proof}
 Define the mapping from the collection of mapping class group orbits
of nested multicurves to $\mathcal{\tilde{A}}_{k}$ by $\mathcal{O}(\Gamma)\mapsto\left\{ \left(c_{i},d_{i},\left\{ a_{1}^{i},\ldots,a_{c_{i}}^{i}\right\} \right)\right\} _{i=1}^{k+1}$
with the collection obtained by cutting along a representative of
the orbit. The mapping is well-defined by Remark \ref{rem:mcg-orbits}.
It is obvious that from this map, that every collection in $\mathcal{\tilde{A}}_{k}$
has at least one multicurve orbit associated to it. Moreover, applying
a permutation to the components of the multicurves in an orbit provides
a distinct orbit giving rise to the same collection. Any other mapping
class group orbit associated to a collection is obtained by some gluing
of the subsurfaces $\Sigma_{0,c_{i},d_{i}}$ with the cusp labels
$\left\{ a_{1}^{i},\ldots,a_{c_{i}}^{i}\right\} $ to obtain $\Sigma_{0,n}$.
The number of such gluings depends only on the number of boundaries
$d_{i}$ on each subsurface component and is clearly bounded by the
number of ways to pair the $2k$ boundaries of the subsurfaces. Note
that this is an over-count since some gluings can give rise to the
same multicurve orbit and some pairings are not permissible as they
do not give rise to $\Sigma_{0,n}$. The number of such pairing is
$(2k-1)!!$ and so when accounting also for the permutations of the
multicurves, we obtain an upper bound of $(2k-1)!!k!$.

\end{proof}
Now for each collection $\left\{ \left(c_{i},d_{i},\left\{ a_{1}^{i},\ldots,a_{c_{i}}^{i}\right\} \right)\right\} _{i=1}^{k+1}\in\tilde{\mathcal{A}}_{k}$
we will fix an arbitrary ordering on the pairs in the collection.
We put this ordering on so that we can refer to specific indexed elements
of a given collection. We will denote by $\mathcal{\tilde{A}}_{k}^{o}$
the collection $\mathcal{\tilde{A}}_{k}$ with a fixed ordering on
the pairs in each of its elements. We then have 
\[
N_{n,[a,b]}^{N,k}(X)=\sum_{\left\{ \left(c_{i},d_{i},\left\{ a_{1}^{i},\ldots,a_{c_{i}}^{i}\right\} \right)\right\} _{i=1}^{k+1}\in\mathcal{\tilde{A}}_{k}^{o}}\sum_{\mathcal{O}(\Gamma)}\sum_{\Gamma=(\gamma_{1},\ldots,\gamma_{n})\in\mathcal{O}(\Gamma)}\ind\left(\frac{a}{\sqrt{n}}\leq\ell_{X}(\gamma_{1}),\ldots,\ell_{X}(\gamma_{k})\leq\frac{b}{\sqrt{n}}\right),
\]
where the middle summation is over all mapping class group orbits
of nested multicurves that are associated with the ordered collection
$\left\{ \left(c_{i},d_{i},\left\{ a_{1}^{i},\ldots,a_{c_{i}}^{i}\right\} \right)\right\} _{i=1}^{k+1}$
which from Lemma \ref{lem:nested-curve-count}, there are at $(2k-1)!!k!$
such orbits.

By Mirzakhani's integration formula (Theorem \ref{thm:MIF}) the expectation
of 
\[
\sum_{\Gamma=(\gamma_{1},\ldots,\gamma_{k})\in\mathcal{O}(\Gamma)}\ind\left(\frac{a}{\sqrt{n}}\leq\ell_{X}(\gamma_{1}),\ldots,\ell_{X}(\gamma_{k})\leq\frac{b}{\sqrt{n}}\right),
\]
 depends only upon the topological decomposition of the cut surface
$X\setminus\Gamma$, the information of which is contained in the
associated collection $\left\{ \left(c_{i},d_{i},\left\{ a_{1}^{i},\ldots,a_{c_{i}}^{i}\right\} \right)\right\} _{i=1}^{k+1}\in\mathcal{\tilde{A}}_{k}^{o}$.
In fact, it is independent of the puncture labels in each triple of
a collection, and so we will consider the collection $\mathcal{A}_{k}^{o}$
which contains the ordered collections $\left\{ \left(c_{i},d_{i}\right)\right\} _{i=1}^{k+1}$
such that $\left\{ \left(c_{i},d_{i},\left\{ a_{1}^{i},\ldots,a_{c_{i}}^{i}\right\} \right)\right\} _{i=1}^{k+1}\in\tilde{\mathcal{A}_{k}^{o}}$
for some partition of the labels $\left\{ 1,\ldots,n\right\} .$ Each
element of $\mathcal{A}_{k}^{o}$ corresponds to precisely 

\[
{n \choose c_{1},\ldots,c_{k}}=\frac{n!}{c_{1}!\cdots c_{k}!(n-\sum_{i=1}^{k}c_{i})!}
\]
elements of $\tilde{\mathcal{A}_{k}^{o}}$. Recall from Section \ref{sec:background}
that since we will only deal with the moduli space of genus zero surfaces,
we use the notations $V_{m}(x_{1},\ldots,x_{m}):=V_{0,m}(x_{1},\ldots,x_{m})$
and $V_{m}:=V_{0,m}$ for the moduli space volumes. We then obtain
the following.
\begin{lem}
\label{lem:nested-curve-exp-ub}For any $k\geq1$,

\[
\mathbb{E}_{n}\left(N_{n,[a,b]}^{N,k}(X)\right)\leq\sum_{\left\{ (c_{i},d_{i})\right\} _{i=1}^{k+1}\in\mathcal{A}_{k}^{o}}(2k-1)!!k!{n \choose c_{1},\ldots,c_{k}}\frac{b^{2k}}{2^{k}}\frac{\prod_{i=1}^{k+1}V_{c_{i}+d_{i}}}{n^{k}V_{n}}\left(1+O\left(\frac{kb^{2}}{n}\right)\right).
\]
\end{lem}

\begin{proof}
By Theorem \ref{thm:MIF} we have
\begin{align*}
 & \mathbb{E}_{n}\left(N_{n,[a,b]}^{N,k}(X)\right)\\
 & =\sum_{\left\{ \left(c_{i},d_{i},\left\{ a_{1}^{i},\ldots,a_{c_{i}}^{i}\right\} \right)\right\} _{i=1}^{k+1}\in\tilde{\mathcal{A}_{k}^{o}}}\sum_{\mathcal{O}(\Gamma)}\frac{1}{V_{n}}\int_{\frac{a}{\sqrt{n}}}^{\frac{b}{\sqrt{n}}}\cdots\int_{\frac{a}{\sqrt{n}}}^{\frac{b}{\sqrt{n}}}x_{1}\cdots x_{k}\prod_{i=1}^{k+1}V_{c_{i}+d_{i}}(\mathbf{0}_{c_{i}},\mathbf{x}_{i})\mathrm{d}x_{1}\cdots\mathrm{d}x_{k},
\end{align*}
where $\mathbf{0}_{a}:=(0,\ldots,0)$ is of length $a$ and the $\mathbf{x}_{i}$
are length $d_{i}$ vectors containing information about the lengths
of the boundaries of the subsurfaces obtained when cutting along each
multicurve whose lengths are assigned $x_{1},\ldots,x_{k}$; so in
particular, each $x_{j}$ appears exactly twice among all of the vectors
$\mathbf{x}_{i}$. Using the upper bound of Lemma \ref{lem:volbds},
the correspondence of Lemma \ref{lem:nested-curve-count} and the
passage from $\mathcal{\tilde{A}}_{k}^{o}$ to $\mathcal{A}_{k}^{o}$,
we obtain
\begin{align*}
\mathbb{E}_{n}\left(N_{n,[a,b]}^{N,k}(X)\right) & \leq\sum_{\left\{ (c_{i},d_{i})\right\} _{i=1}^{k+1}\in\mathcal{A}_{k}^{o}}(2k-1)!!k!{n \choose c_{1},\ldots,c_{k}}\frac{\prod_{i=1}^{k+1}V_{c_{i}+d_{i}}}{V_{n}}\left(\int_{\frac{a}{\sqrt{n}}}^{\frac{b}{\sqrt{n}}}4\frac{\mathrm{sinh^{2}}\left(\frac{x}{2}\right)}{x}dx\right)^{k}\\
 & \leq\sum_{\left\{ (c_{i},d_{i})\right\} _{i=1}^{k+1}\in\mathcal{A}_{k}^{o}}(2k-1)!!k!{n \choose c_{1},\ldots,c_{k}}\frac{b^{2k}}{2^{k}}\frac{\prod_{i=1}^{k+1}V_{c_{i}+d_{i}}}{n^{k}V_{n}}\left(1+O\left(\frac{kb^{2}}{n}\right)\right).
\end{align*}
\end{proof}
We prove the following estimate for this volume summation.
\begin{prop}
\label{prop:Nested contribution}For any $k\geq1$,
\[
\mathbb{E}_{n}\left(N_{n,[a,b]}^{N,k}(X)\right)=O_{k}\left(\frac{b^{2k}}{n}\right).
\]
\end{prop}

\begin{proof}
We will show that 
\[
\sum_{\left\{ (c_{i},d_{i})\right\} _{i=1}^{k+1}\in\mathcal{A}_{k}^{o}}{n \choose c_{1},\ldots,c_{k}}\frac{\prod_{i=1}^{k+1}V_{c_{i}+d_{i}}}{n^{k}V_{n}}=O_{k}\left(\frac{1}{n}\right),
\]
and then the result follows from Lemma \ref{lem:nested-curve-exp-ub}.
Note that 
\begin{align}
\sum_{\left\{ (c_{i},d_{i})\right\} _{i=1}^{k+1}\in\mathcal{A}_{k}^{o}}{n \choose c_{1},\ldots,c_{k}}\frac{\prod_{i=1}^{k+1}V_{c_{i}+d_{i}}}{n^{k}V_{n}} & \leq\sum_{t=1}^{k+1}\sum_{\substack{\left\{ (c_{i},d_{i})\right\} _{i=1}^{k+1}\in\mathcal{A}_{k}^{o}\\
c_{t}\geq\frac{n}{k+1}
}
}\frac{n!V_{c_{t}+d_{t}}}{c_{t}!n^{k}V_{n}}\prod_{\substack{i=1\\
i\neq t
}
}^{k+1}\frac{V_{c_{i}+d_{i}}}{c_{i}!},\label{eq:sumbd}
\end{align}
since $\sum_{i=1}^{k+1}c_{i}=n$, so at least one $c_{i}$ must be
at least $\frac{n}{k+1}$. Now if $c_{t}\geq\frac{n}{k+1}$, then
for $n$ sufficiently large, by Stirling's approximation for the factorial,

\[
\frac{(c_{t}+d_{t})!}{n^{k}c_{t}!}\leq\frac{(c_{t}+d_{t})^{c_{t}+d_{t}+\frac{1}{2}}}{n^{k}c_{t}^{c_{t}+\frac{1}{2}}}e^{-d_{t}}=e^{-d_{t}}\frac{c_{t}^{d_{t}}}{n^{k}}\left(1+\frac{d_{t}}{c_{t}}\right)^{c_{t}+d_{t}+\frac{1}{2}}\leq n^{d_{t}-k}\left(1+\frac{k(k+1)}{n}\right)^{k+\frac{1}{2}},
\]
using the fact that $\frac{n}{k+1}\leq c_{t}\leq n$ and $d_{t}\leq k.$
Thus for $n$ sufficiently large, Theorem \ref{thm:ZografManin} gives
\[
\frac{n!V_{c_{t}+d_{t}}}{c_{t}!n^{k}V_{n}}\leq C_{1}\frac{n^{\frac{7}{2}}}{(c_{t}+d_{t})^{\frac{7}{2}}}\frac{(c_{t}+d_{t})!}{n^{k}c_{t}!}\left(\frac{x_{0}}{2\pi^{2}}\right)^{n-(c_{t}+d_{t})}\leq C_{2}n^{d_{t}-k}\left(\frac{x_{0}}{2\pi^{2}}\right)^{n-(c_{t}+d_{t})},
\]
for some constants $C_{1},C_{2}>0$ possibly dependent upon $k$.
Next note that from Lemma \ref{lem:mirzakhani-vol-comparison}, there
exists a universal constant $\gamma>0$ such that $V_{m+1}\leq\gamma(m-2)V_{m}$
for all $m\geq3.$ Thus, for $c_{i}\geq4$ (since $d_{i}\geq1$)

\[
\frac{V_{c_{i}+d_{i}}}{c_{i}!}\leq\gamma^{d_{i}}\frac{(c_{i}+d_{i}-3)!}{c_{i}!(c_{i}-3)!}V_{c_{i}}.
\]
But, by Stirling's approximation, 
\[
\frac{(c_{i}+d_{i}-3)!}{(c_{i}-3)!}\leq e^{-d_{i}}\frac{(c_{i}+d_{i}-3)^{c_{i}+d_{i}-\frac{5}{2}}}{(c_{i}-3)^{c_{i}-\frac{5}{2}}}\leq e^{-d_{i}}c_{i}^{d_{i}}(1+d_{i})^{d_{i}+\frac{1}{2}}\left(1+\frac{d_{i}}{c_{i}-3}\right)^{c_{i}-3}\leq c_{i}^{d_{i}}(1+k)^{d_{i}+\frac{1}{2}}.
\]
Thus,
\[
\frac{V_{c_{i}+d_{i}}}{c_{i}!}\leq\gamma^{d_{i}}(1+k)^{d_{i}+\frac{1}{2}}\frac{c_{i}^{d_{i}}V_{c_{i}}}{c_{i}!}.
\]
Moreover, if one sets $V_{0}=V_{1}=V_{2}=1,$then for $c_{i}\leq3$,
we also obtain
\[
\frac{V_{c_{i}+d_{i}}}{c_{i}!}\leq\gamma^{d_{i}}(c_{i}+d_{i}-3)!\frac{V_{3}}{c_{i}!}\leq\gamma^{d_{i}}d_{i}!\frac{V_{c_{i}}}{c_{i}!}\leq\gamma^{d_{i}}(1+k)^{d_{i}+\frac{1}{2}}\frac{\max\left\{ 1,c_{i}^{d_{i}}\right\} V_{c_{i}}}{c_{i}!}.
\]
Returning to the bound in equation (\ref{eq:sumbd}), for some constant
$C>0$ dependent only upon $k$ we have 
\begin{align*}
 & \sum_{t=1}^{k+1}\sum_{\substack{\left\{ (c_{i},d_{i})\right\} _{i=1}^{k+1}\in\mathcal{A}_{k}^{o}\\
c_{t}\geq\frac{n}{k+1}
}
}\frac{n!V_{c_{t}+d_{t}}}{c_{t}!n^{k}V_{n}}\prod_{\substack{i=1\\
i\neq t
}
}^{k}\frac{V_{c_{i}+d_{i}}}{c_{i}!}\\
 & \leq C\sum_{t=1}^{k+1}\sum_{\substack{\left\{ (c_{i},d_{i})\right\} _{i=1}^{k+1}\in\mathcal{A}_{k}^{o}\\
c_{t}\geq\frac{n}{k+1}
}
}n^{d_{t}-k}\left(\frac{x_{0}}{2\pi^{2}}\right)^{n-(c_{t}+d_{t})}\prod_{\substack{i=1\\
i\neq t
}
}^{k+1}\frac{\max\left\{ 1,c_{i}^{d_{i}}\right\} V_{c_{i}}}{c_{i}!}\\
 & \leq C\sum_{t=1}^{k+1}\sum_{\substack{\left\{ (c_{i},d_{i})\right\} _{i=1}^{k+1}\in\mathcal{A}_{k}^{o}\\
c_{t}\geq\frac{n}{k+1}
}
}n^{d_{t}-k+\sum_{\substack{j=1\\
j\neq t
}
}^{k+1}(d_{j}-2)\ind(d_{j}\geq2)}\prod_{\substack{i=1\\
i\neq t
}
}^{k+1}\frac{\max\left\{ 1,c_{i}^{2}\right\} V_{c_{i}}x_{0}^{c_{i}}}{c_{i}!(2\pi^{2})^{c_{i}}},
\end{align*}
where we use $\left(\frac{x_{0}}{2\pi^{2}}\right)^{n-c_{t}}=\left(\frac{x_{0}}{2\pi^{2}}\right)^{\sum_{\substack{i=1,i\neq t}
}^{k}c_{i}}$ and absorb the constant $\left(\frac{x_{0}}{2\pi^{2}}\right)^{-d_{t}}$
into the constant $C$ since $d_{t}$ is bounded by $k$, and $c_{i}^{d_{i}}=c_{i}^{2}c_{i}^{d_{i}-2}\leq c_{i}^{2}n^{(d_{i}-2)\ind(d_{i}\geq2)}$.
Note that because $d_{j}<k$ for each $j$, we have
\[
d_{t}-k+\sum_{\substack{j=1\\
j\neq t
}
}^{k+1}(d_{j}-2)\ind(d_{j}\geq2)\leq-1.
\]
Indeed, let $\mathcal{I\subseteq}\{j\in\{1,\ldots,k+1\}\setminus\{t\}:d_{j}=1\}$.
Then, 
\begin{align*}
d_{t}-k+\sum_{\substack{j=1\\
j\neq t
}
}^{k+1}(d_{j}-2)\ind(d_{j}\geq2) & =-k+\sum_{j=1}^{k+1}d_{j}-\sum_{j\in\mathcal{I}}d_{j}-\sum_{\substack{j\notin\mathcal{I}\\
j\neq t
}
}2=|\mathcal{I}|-k,
\end{align*}
but $|\mathcal{I}|\leq k-1$ since if $d_{j}=1$ for every $j\neq t$,
then $d_{t}=k$, which is a contradiction to a collection being in
$\mathcal{A}_{k}$ and hence the result follows. Next observe that
in an ordered collection $\{(c_{i},d_{i})\}_{i=1}^{k+1}$, the $t^{\mathrm{th}}$
pair is determined entirely by the other pairs since $c_{t}=n-\sum_{i=1,i\neq t}^{k+1}c_{i}$
and $d_{t}=2k-\sum_{i=1,i\neq t}^{k+1}d_{i}$. Thus, we see that $\mathcal{A}_{k}^{o}$
is contained in the set 
\[
\left\{ \{(c_{i},d_{i})\}_{i=1}^{k+1}:c_{i}\in\mathbb{N},1\leqslant d_{i}\leqslant k,c_{t}=\max\left\{ 0,n-\sum_{i=1,i\neq t}^{k+1}c_{i}\right\} ,d_{t}=\max\left\{ 0,2k-\sum_{i=1,i\neq t}^{k+1}d_{i}\right\} \right\} .
\]
Thus, 
\begin{align*}
 & C\sum_{t=1}^{k+1}\sum_{\substack{\left\{ (c_{i},d_{i})\right\} _{i=1}^{k+1}\in\mathcal{A}_{k}^{o}\\
c_{t}\geq\frac{n}{k+1}
}
}n^{d_{t}-k+\sum_{\substack{j=1\\
j\neq t
}
}^{k+1}(d_{j}-2)\ind(d_{j}\geq2)}\prod_{\substack{i=1\\
i\neq t
}
}^{k+1}\frac{\max\left\{ 1,c_{i}^{2}\right\} V_{c_{i}}x_{0}^{c_{t}}}{c_{i}!(2\pi^{2})^{c_{i}}}\\
 & \leq\frac{C}{n}\sum_{t=1}^{k+1}\sum_{\left\{ (c_{i},d_{i})\right\} _{i=1}^{k+1}\in\mathcal{A}_{k}^{o}}\prod_{\substack{i=1\\
i\neq t
}
}^{k+1}\frac{\max\left\{ 1,c_{i}^{2}\right\} V_{c_{i}}x_{0}^{c_{i}}}{c_{i}!(2\pi^{2})^{c_{i}}}\\
 & \leq\frac{C}{n}\sum_{t=1}^{k+1}\underbrace{\sum_{d_{1}=1}^{k}\cdots\sum_{d_{k+1}=1}^{k}}_{\text{no \ensuremath{d_{t}} term}}\underbrace{\sum_{c_{i}=0}^{\infty}\cdots\sum_{c_{k+1}=0}^{\infty}}_{\text{no \ensuremath{c_{t}} term}}\prod_{\substack{i=1\\
i\neq t
}
}^{k+1}\frac{\max\left\{ 1,c_{i}^{2}\right\} V_{c_{i}}x_{0}^{c_{i}}}{c_{i}!(2\pi^{2})^{c_{i}}}=\frac{B}{n}\left(1+\sum_{c=1}^{\infty}\frac{c^{2}V_{c}x_{0}^{c}}{c!(2\pi^{2})^{c}}\right)^{k},
\end{align*}
where $B>0$ is some constant dependent only upon $k$. The latter
summation then converges due to Theorem \ref{thm:ZografManin} and
hence we obtain the desired result.
\end{proof}

\subsection{Contribution of unnested curves}

\label{subsec:unnested}

We now compute the contribution of $N_{n,[a,b]}^{U,k}(X)$ to the
expectation. We start with an enumeration system of the multicurves
as was done for the nested multicurves. For $k\geq1,$ we let $\mathcal{B}_{k}$
be the collection of ordered tuples $\left(\left(c_{i},d_{i},\left\{ a_{1}^{i},\ldots,a_{c_{i}}^{i}\right\} \right)\right)_{i=1}^{k+1}$
satisfying 
\begin{enumerate}
\item $\sum_{i=1}^{k+1}c_{i}=n$ and $c_{i}\geq2$ for $1\leq i\leq k$,
\item $d_{1}=\ldots=d_{k}=1$ and $d_{k+1}=k$,
\item $\{a_{1}^{i},\ldots,a_{c_{i}}^{i}\}\subseteq\{1,\ldots,n\}$ are pairwise
disjoint and their union over all $i$ is $\{1,\ldots,n\}$,
\item $c_{i}+d_{i}\geq3$.
\end{enumerate}
\begin{lem}
\label{lem:unnested-multicurve-decomp}For $k\geq2,$ there is a bijection
between mapping class group orbits of unnested multicurves and $\mathcal{B}_{k}$.
\end{lem}

\begin{proof}
We define a bijective mapping in the following way. Cutting $\Sigma_{n}$
along a representative $\Gamma=\left(\gamma_{1},\ldots,\gamma_{k}\right)$
of a mapping class group orbit of an unnested multicurve gives a decomposition
of $\Sigma_{n}$ as
\[
\Sigma_{n}\setminus\Gamma=\left(\bigsqcup_{i=1}^{k}\Sigma_{c_{i},1}\right)\sqcup\Sigma_{c_{k+1},k},
\]
where there are labels $\left\{ a_{1}^{i},\ldots,a_{c_{i}}^{i}\right\} $
on the $c_{i}$ punctures of $\Sigma_{c_{i},d_{i}}$ inherited from
the labeling of $\Sigma_{n}$, with the topological decomposition
being independent of the orbit representative chosen by Remark \ref{rem:mcg-orbits}.
Indeed, cutting $\Sigma_{n}$ along $\gamma_{i}$ separates $\Sigma_{n}$
into two subsurfaces $\Sigma_{c_{i},1}\sqcup\Sigma_{n-c_{i},1}$ with
one of these subsurfaces containing no other component of $\Gamma$
in its interior, which we take to be $\Sigma_{c_{i},1}$. The cutting
of the remainder of the surfaces along any other curve component leaves
$\Sigma_{c_{i},1}$unchanged, and so repeating this procedure recursively
for each of the curve components gives the stated decomposition. So,
for $1\leq i\leq k$, $c_{i}$ is the number of cusps on the subsurface
of $X$ obtained from cutting along $\gamma_{i}$ that does not contain
any other curve component of $\Gamma$. The multicurve is then identified
with the tuple $\left(\left(c_{i},d_{i},\left\{ a_{1}^{i},\ldots,a_{c_{i}}^{i}\right\} \right)\right)_{i=1}^{k+1}$
.

Surjectivity of this identification is clear. For injectivity, we
note that if $\mathcal{O}(\Gamma)$ and $\mathcal{O}(\Gamma')$ are
two unnested multicurve mapping class orbits with the same decomposition,
then taking representatives $\Gamma=(\gamma_{1},\ldots,\gamma_{k})\in\mathcal{O}(\Gamma)$
and $\Gamma'=(\gamma_{1}',\ldots,\gamma_{k}')\in\mathcal{O}(\Gamma')$
we see that for each $1\leqslant i\leqslant k$, $\Sigma_{n}\setminus\gamma_{i}$
and $\Sigma_{n}\setminus\gamma_{i}'$ are puncture label preserving
homeomorphic by construction of the mapping. This means that they
are in the same mapping class group orbit and hence the multicurves
are also using Remark \ref{rem:mcg-orbits}. 
\end{proof}
In the case of $k=1$, we instead have the following. 
\begin{lem}
\label{lem:single-curve}There is a $1-2$ correspondence between
mapping class group orbits of unnested multicurves and $\mathcal{B}_{1}$.
\end{lem}

\begin{proof}
We construct the association as in the proof of Lemma \ref{lem:unnested-multicurve-decomp},
the difference is that after cutting the labeled surface, there is
no canonical choice for the ordered tuple that is obtained. We thus
associate both choices with the single curve, namely $\left(\left(c_{1},1,\left\{ a_{1}^{1},\ldots,a_{c_{1}}^{1}\right\} \right),\left(c_{2},1,\left\{ a_{1}^{2},\ldots,a_{c_{2}}^{2}\right\} \right)\right)$
and $\left(\left(c_{2},1,\left\{ a_{1}^{2},\ldots,a_{c_{2}}^{2}\right\} \right),\left(c_{1},1,\left\{ a_{1}^{1},\ldots,a_{c_{1}}^{1}\right\} \right)\right)$.
\end{proof}
Using this combinatorial interpretation of the mapping class group
orbits, we obtain the following estimate. As before, we shall use
the notation $V_{m}(x_{1},\ldots,x_{m}):=V_{0,m}(x_{1},\ldots,x_{m})$
and $V_{m}:=V_{0,m}$.
\begin{lem}
\label{lem:unnested-exp-ub}We have the following estimates,
\[
\mathbb{E}_{n}\left(N_{n,[a,b]}^{U,1}(X)\right)=\sum_{c=2}^{\left\lfloor \frac{n}{2}\right\rfloor }{n \choose c}\left(\frac{b^{2}-a^{2}}{2}\right)\frac{1}{n}\frac{V_{c+1}V_{n-c+1}}{V_{n}}\left(1+O\left(\frac{b^{2}}{n}\right)\right),
\]

\begin{align*}
\mathbb{E}_{n}\left(N_{n,[a,b]}^{U,2}(X)\right) & =\sum_{\substack{\substack{(c_{1},}
c_{2})\in\mathbb{N}^{2},\\
c_{1}+c_{2}\leq n-1,\\
c_{i}\geq2
}
}{n \choose c_{1},c_{2}}\left(\frac{b^{2}-a^{2}}{2}\right)^{2}\frac{1}{n^{2}}\frac{V_{c_{1}+1}V_{c_{2}+1}V_{n-c_{1}-c_{2}+2}}{V_{n}}\left(1+O\left(\frac{b^{2}}{n}\right)\right),
\end{align*}
and for $k\geq3$,

\begin{align*}
 & \mathbb{E}_{n}\left(N_{n,[a,b]}^{U,k}(X)\right)\\
 & =\sum_{\substack{(c_{1},\ldots,c_{k})\in\mathbb{N}^{k},\\
\sum_{i=1}^{k}c_{i}\leq n,\\
c_{i}\geq2.
}
}{n \choose c_{1},\ldots,c_{k}}\left(\frac{b^{2}-a^{2}}{2}\right)^{k}\frac{1}{n^{k}}\frac{V_{c_{1}+1}\cdots V_{c_{k}+1}V_{n-\sum_{i=1}^{k}c_{i}+k}}{V_{n}}\left(1+O_{k}\left(\frac{b^{2}}{n}\right)\right).
\end{align*}
\end{lem}

\begin{proof}
Notice that by Lemmas \ref{lem:unnested-multicurve-decomp} and \ref{lem:single-curve}
we have

\begin{align*}
 & \mathbb{E}_{n}\left(N_{n,[a,b]}^{U,k}(X)\right)\\
 & =C_{k}\sum_{\left(\left(c_{i},d_{i},\left\{ a_{1}^{i},\ldots,a_{c_{i}}^{i}\right\} \right)\right)_{i=1}^{k+1}\in\mathcal{B}_{k}}\mathbb{E}_{n}\sum_{\Gamma=(\gamma_{1},\ldots,\gamma_{k})\in\mathcal{O}(\Gamma)}\ind\left(\frac{a}{\sqrt{n}}\leq\ell_{X}(\gamma_{1}),\ldots,\ell_{X}(\gamma_{k})\leq\frac{b}{\sqrt{n}}\right),
\end{align*}
where the mapping class group orbit $\mathcal{O}(\Gamma)$ is the
unique such orbit associated with a given ordered collection $\left(\left(c_{i},d_{i},\left\{ a_{1}^{i},\ldots,a_{c_{i}}^{i}\right\} \right)\right)_{i=1}^{k+1}$,
$C_{1}=\frac{1}{2}$ and $C_{k}=1$ for $k\geq2$. By Theorem \ref{thm:MIF},
we have 

\begin{align*}
 & \mathbb{E}_{n}\sum_{\Gamma=(\gamma_{1},\ldots,\gamma_{k})\in\mathcal{O}(\Gamma)}\ind\left(\frac{a}{\sqrt{n}}\leq\ell_{X}(\gamma_{1}),\ldots,\ell_{X}(\gamma_{k})\leq\frac{b}{\sqrt{n}}\right)\\
 & =\frac{1}{V_{n}}\int_{\frac{a}{\sqrt{n}}}^{\frac{b}{\sqrt{n}}}\cdots\int_{\frac{a}{\sqrt{n}}}^{\frac{b}{\sqrt{n}}}V_{c_{k+1}+k}(\boldsymbol{0}_{c_{k+1}},x_{1},\ldots,x_{k})\prod_{i=1}^{k}x_{i}V_{c_{i}+1}(\boldsymbol{0}_{c_{i}},x_{i})\mathrm{d}x_{1}\cdots\mathrm{d}x_{k}.
\end{align*}
By Lemma \ref{lem:volbds} we have 
\begin{align*}
V_{c_{i}+1} & \leqslant V_{c_{i}+1}(\boldsymbol{0}_{c_{i}},x_{i})\leqslant V_{c_{i}+1}\frac{\sinh\left(\frac{x_{i}}{2}\right)}{\left(\frac{x_{i}}{2}\right)},\\
V_{c_{k+1}+k} & \leqslant V_{c_{k+1}+k}(\boldsymbol{0}_{c_{k+1}},x_{1},\ldots,x_{k})\leqslant V_{c_{k+1}+k}\prod_{i=1}^{k}\frac{\sinh\left(\frac{x_{i}}{2}\right)}{\left(\frac{x_{i}}{2}\right)}.
\end{align*}
It then follows from the Taylor expansion 
\[
4\frac{\sinh^{2}\left(\frac{x_{i}}{2}\right)}{x_{i}}=x_{i}+O\left(x_{i}^{3}\right)=x_{i}+O\left(\frac{b^{3}}{n^{\frac{3}{2}}}\right),
\]
that when $x_{i}\leqslant\frac{b}{\sqrt{n}}$ we have
\[
V_{c_{k+1}+k}(\boldsymbol{0}_{c_{k+1}},x_{1},\ldots,x_{k})\prod_{i=1}^{k}x_{i}V_{c_{i}+1}(\boldsymbol{0}_{c_{i}},x_{i})=\left(\prod_{i=1}^{k}x_{i}\right)\left(1+O\left(\frac{b^{2}}{n}\right)\right)^{k}.
\]
We then see that 

\begin{align*}
 & \mathbb{E}_{n}\sum_{\Gamma=(\gamma_{1},\ldots,\gamma_{k})\in\mathcal{O}(\Gamma)}\ind\left(\frac{a}{\sqrt{n}}\leq\ell_{X}(\gamma_{1}),\ldots,\ell_{X}(\gamma_{k})\leq\frac{b}{\sqrt{n}}\right)\\
 & =\left(\frac{b^{2}-a^{2}}{2}\right)^{k}\frac{1}{n^{k}}\frac{V_{c_{1}+1}\cdots V_{c_{k}+1}V_{c_{k+1}+k}}{V_{n}}\left(1+O_{k}\left(\frac{b^{2}}{n}\right)\right),
\end{align*}
By definition of the collection $\mathcal{B}_{k}$ we have $c_{k+1}=n-\sum_{i=1}^{k}c_{i}$
and so

\begin{align*}
 & \mathbb{E}_{n}\left(N_{n,[a,b]}^{U,k}(X)\right)\\
 & =C_{k}\sum_{\left(\left(c_{i},d_{i},\left\{ a_{1}^{i},\ldots,a_{c_{i}}^{i}\right\} \right)\right)_{i=1}^{k+1}\in\mathcal{B}_{k}}\left(\frac{b^{2}-a^{2}}{2}\right)^{k}\frac{1}{n^{k}}\frac{V_{c_{1}+1}\cdots V_{c_{k}+1}V_{n-\sum_{i=1}^{k}c_{i}+k}}{V_{n}}\left(1+O_{k}\left(\frac{b^{2}}{n}\right)\right),
\end{align*}
where the summand is independent of the puncture labeling. When the
collection of tuples $((c_{i},1))_{i=1}^{k}$ are fixed, there is
a single choice of value for $c_{k+1}$ and ${n \choose c_{1},\ldots,c_{k}}$
choices of assigning the cusp labels to the subsurfaces according
to the definition of $\mathcal{B}_{k}$. In the case of $k\geq3$
the summation is thus equal to

\begin{align*}
 & \mathbb{E}_{n}\left(N_{n,[a,b]}^{U,k}(X)\right)\\
 & =\sum_{\substack{(c_{1},\ldots,c_{k})\in\mathbb{N}^{k},\\
\sum_{i=1}^{k}c_{i}\leq n,\\
c_{i}\geq2,
}
}{n \choose c_{1},\ldots,c_{k}}\left(\frac{b^{2}-a^{2}}{2}\right)^{k}\frac{1}{n^{k}}\frac{V_{c_{1}+1}\cdots V_{c_{k}+1}V_{n-\sum_{i=1}^{k}c_{i}+k}}{V_{n}}\left(1+O_{k}\left(\frac{b^{2}}{n}\right)\right).
\end{align*}
For $k=2,$ the requirement on the indices that $c_{k+1}+2\geq3,$
means the indexing runs over the replaced condition $c_{1}+c_{2}\leq n-1$.
Similarly for $k=1,$by Lemma \ref{lem:single-curve},
\begin{align*}
\mathbb{E}_{n}\left(N_{n,[a,b]}^{U}(X)\right) & =\frac{1}{2}\sum_{\left(\left(c_{i},d_{i},\left\{ a_{1}^{i},\ldots,a_{c_{i}}^{i}\right\} \right)\right)_{i=1}^{2}\in\mathcal{B}_{1}}{n \choose c_{1}}\left(\frac{b^{2}-a^{2}}{2}\right)\frac{1}{n}\frac{V_{c_{1}+1}V_{n-c_{1}+1}}{V_{n}}\left(1+O\left(\frac{b^{2}}{n}\right)\right)\\
 & =\sum_{c=2}^{\left\lfloor \frac{n}{2}\right\rfloor }{n \choose c}\left(\frac{b^{2}-a^{2}}{2}\right)\frac{1}{n}\frac{V_{c+1}V_{n-c+1}}{V_{n}}\left(1+O\left(\frac{b^{2}}{n}\right)\right).
\end{align*}
\end{proof}
We now determine the leading order of this summand.
\begin{prop}
\label{prop:Unnested contribution}Define 
\[
\alpha=\sum_{i=2}^{\infty}\frac{V_{i+1}}{i!}\left(\frac{x_{0}}{2\pi^{2}}\right)^{i-1},
\]
then,

\[
\mathbb{E}_{n}\left(N_{n,[a,b]}^{U,k}(X)\right)=\left(\frac{b^{2}-a^{2}}{2}\right)^{k}\alpha^{k}\left(1+O_{k}\left(\frac{1}{n^{\frac{1}{2}}}\right)\right).
\]
\end{prop}

\begin{proof}
We first consider the case when $k\geq3$, the result is similar for
$k=1$ and $2$. By Lemma \ref{lem:unnested-exp-ub}, it suffices
to show that
\[
\sum_{\substack{(c_{1},\ldots,c_{k})\in\mathbb{N}^{k},\\
\sum_{i=1}^{k}c_{i}\leq n,\\
c_{i}\geq2.
}
}{n \choose c_{1},\ldots,c_{k}}\frac{1}{n^{k}}\frac{V_{c_{1}+1}\cdots V_{c_{k}+1}V_{n-\sum_{i=1}^{k}c_{i}+k}}{V_{n}}=\alpha^{k}+O_{k}\left(\frac{1}{\sqrt{n}}\right).
\]
 Observe first that 
\begin{align*}
\alpha^{k} & =\sum_{c_{1}=2}^{\infty}\cdots\sum_{c_{k}=2}^{\infty}\left(\frac{x_{0}}{2\pi^{2}}\right)^{\sum_{i=1}^{k}\left(c_{i}-1\right)}\prod_{i=1}^{k}\frac{V_{c_{1}+1}}{c_{i}!}\\
 & =\sum_{c_{1}=2}^{\lfloor n^{\frac{1}{2}}\rfloor}\cdots\sum_{c_{k}=2}^{\lfloor n^{\frac{1}{2}}\rfloor}\left(\frac{x_{0}}{2\pi^{2}}\right)^{\sum_{i=1}^{k}\left(c_{i}-1\right)}\prod_{i=1}^{k}\frac{V_{c_{1}+1}}{c_{i}!}+\sum_{\substack{(c_{1},\ldots,c_{k})\in\mathbb{N}^{k}\\
\exists\,i\,:c_{i}\geq\lceil n^{\frac{1}{2}}\rceil\\
c_{j}\geq2
}
}\left(\frac{x_{0}}{2\pi^{2}}\right)^{\sum_{i=1}^{k}\left(c_{i}-1\right)}\prod_{i=1}^{k}\frac{V_{c_{1}+1}}{c_{i}!}\\
 & \leq\sum_{c_{1}=2}^{\lfloor n^{\frac{1}{2}}\rfloor}\cdots\sum_{c_{k}=2}^{\lfloor n^{\frac{1}{2}}\rfloor}\left(\frac{x_{0}}{2\pi^{2}}\right)^{\sum_{i=1}^{k}\left(c_{i}-1\right)}\prod_{i=1}^{k}\frac{V_{c_{1}+1}}{c_{i}!}+k\sum_{c_{1}=\lceil n^{\frac{1}{2}}\rceil}^{\infty}\sum_{c_{2}=2}^{\infty}\cdots\sum_{c_{k}=2}^{\infty}\left(\frac{x_{0}}{2\pi^{2}}\right)^{\sum_{i=1}^{k}\left(c_{i}-1\right)}\prod_{i=1}^{k}\frac{V_{c_{1}+1}}{c_{i}!}.
\end{align*}
Thus, 
\begin{align*}
 & \left|\sum_{\substack{(c_{1},\ldots,c_{k})\in\mathbb{N}^{k}\\
\sum_{i=1}^{k}c_{i}\leq n\\
c_{i}\geq2
}
}{n \choose c_{1},\ldots,c_{k}}\frac{1}{n^{k}}\frac{V_{c_{1}+1}\cdots V_{c_{k}+1}V_{n-\sum_{i=1}^{k}c_{i}+k}}{V_{n}}-\alpha^{k}\right|\\
\leqslant & \left|\sum_{\substack{(c_{1},\ldots,c_{k})\in\mathbb{N}^{k}\\
2\leq c_{i}\leq n^{\frac{1}{2}}
}
}{n \choose c_{1},\ldots,c_{k}}\frac{1}{n^{k}}\frac{V_{c_{1}+1}\cdots V_{c_{k}+1}V_{n-\sum_{i=1}^{k}c_{i}+k}}{V_{n}}-\alpha^{k}\right|\\
 & +\left|\sum_{\substack{(c_{1},\ldots,c_{k})\in\mathbb{N}^{k}\\
\exists\,i\,:\,c_{i}\geq n^{\frac{1}{2}}\\
c_{i}\geq2\\
\sum_{i=1}^{k}c_{i}\leq n
}
}{n \choose c_{1},\ldots,c_{k}}\frac{1}{n^{k}}\frac{V_{c_{1}+1}\cdots V_{c_{k}+1}V_{n-\sum_{i=1}^{k}c_{i}+k}}{V_{n}}\right|\\
\leqslant & \underbrace{\left|\sum_{\substack{(c_{1},\ldots,c_{k})\in\mathbb{N}^{k}\\
2\leq c_{i}\leq n^{\frac{1}{2}}
}
}{n \choose c_{1},\ldots,c_{k}}\frac{1}{n^{k}}\frac{V_{c_{1}+1}\cdots V_{c_{k}+1}V_{n-\sum_{i=1}^{k}c_{i}+k}}{V_{n}}-\sum_{c_{1}=2}^{\lfloor n^{\frac{1}{2}}\rfloor}\cdots\sum_{c_{k}=2}^{\lfloor n^{\frac{1}{2}}\rfloor}\left(\frac{x_{0}}{2\pi^{2}}\right)^{\sum_{i=1}^{k}\left(c_{i}-1\right)}\prod_{i=1}^{k}\frac{V_{c_{1}+1}}{c_{i}!}\right|}_{(a)}\\
 & +\underbrace{k\sum_{c_{1}=\lceil n^{\frac{1}{2}}\rceil}^{\infty}\sum_{c_{2}=2}^{\infty}\cdots\sum_{c_{k}=2}^{\infty}\left(\frac{x_{0}}{2\pi^{2}}\right)^{\sum_{i=1}^{k}\left(c_{i}-1\right)}\prod_{i=1}^{k}\frac{V_{c_{1}+1}}{c_{i}!}}_{(b)}+\underbrace{\sum_{\substack{(c_{1},\ldots,c_{k})\in\mathbb{N}^{k}\\
\exists\,i\,:\,c_{i}\geq n^{\frac{1}{2}}\\
c_{i}\geq2\\
\sum_{i=1}^{k}c_{i}\leq n
}
}{n \choose c_{1},\ldots,c_{k}}\frac{1}{n^{k}}\frac{V_{c_{1}+1}\cdots V_{c_{k}+1}V_{n-\sum_{i=1}^{k}c_{i}+k}}{V_{n}}}_{(c)}.
\end{align*}
We first bound the contribution of $(a)$. We calculate that 
\begin{align*}
 & \left|\sum_{\substack{(c_{1},\ldots,c_{k})\in\mathbb{N}^{k}\\
2\leq c_{i}\leq n^{\frac{1}{2}}
}
}{n \choose c_{1},\ldots,c_{k}}\frac{1}{n^{k}}\frac{V_{c_{1}+1}\cdots V_{c_{k}+1}V_{n-\sum_{i=1}^{k}c_{i}+k}}{V_{n}}-\sum_{c_{1}=2}^{\lfloor n^{\frac{1}{2}}\rfloor}\cdots\sum_{c_{k}=2}^{\lfloor n^{\frac{1}{2}}\rfloor}\left(\frac{x_{0}}{2\pi^{2}}\right)^{\sum_{i=1}^{k}\left(c_{i}-1\right)}\prod_{i=1}^{k}\frac{V_{c_{1}+1}}{c_{i}!}\right|\\
= & \left|\sum_{c_{1}=2}^{\lfloor n^{\frac{1}{2}}\rfloor}\cdots\sum_{c_{k}=2}^{\lfloor n^{\frac{1}{2}}\rfloor}\left(\frac{V_{c_{1}+1}\cdot\cdots\cdot V_{c_{k}+1}\cdot V_{n-\sum_{i=1}^{k}c_{i}+k}}{c_{1}!\cdot\cdots\cdot c_{k}!\left(n-\sum_{i=1}^{k}c_{i}\right)!}\frac{n!}{n^{k}V_{n}}-\frac{V_{c_{1}+1}\cdot\cdots\cdot V_{c_{k}+1}}{c_{1}!\cdot\cdots\cdot c_{k}!}\left(\frac{x_{0}}{2\pi^{2}}\right)^{\sum_{i=1}^{k}\left(c_{i}-1\right)}\right)\right|\\
\leqslant & \sum_{c_{1}=2}^{\lfloor n^{\frac{1}{2}}\rfloor}\cdots\sum_{c_{k}=2}^{\lfloor n^{\frac{1}{2}}\rfloor}\frac{V_{c_{1}+1}\cdot\cdots\cdot V_{c_{k}+1}}{c_{1}!\cdot\cdots\cdot c_{k}!}\left|\frac{n!V_{n-\sum_{i=1}^{k}c_{i}+k}}{n^{k}V_{n}\left(n-\sum_{i=1}^{k}c_{i}\right)!}-\left(\frac{x_{0}}{2\pi^{2}}\right)^{\sum_{i=1}^{k}\left(c_{i}-1\right)}\right|.
\end{align*}
Since each $c_{i}\leqslant n^{\frac{1}{2}},$ we can calculate using
the volume asymptotics of Theorem \ref{thm:ZografManin} that 
\begin{align*}
 & \left|\frac{n!V_{n-\sum_{i=1}^{k}c_{i}+k}}{n^{k}V_{n}\left(n-\sum_{i=1}^{k}c_{i}\right)!}-\left(\frac{x_{0}}{2\pi^{2}}\right)^{\sum_{i=1}^{k}\left(c_{i}-1\right)}\right|\\
= & \left|\frac{n^{\frac{7}{2}}\left(n-\sum_{i=1}^{k}c_{i}+k\right)!x_{0}^{n}(2\pi^{2})^{n-\sum_{i=1}^{k}c_{i}+k-3}}{n^{k}\left(n-\sum_{i=1}^{k}c_{i}+k\right)^{\frac{7}{2}}\left(n-\sum_{i=1}^{k}c_{i}\right)!x_{0}^{n-\sum_{i=1}^{k}c_{i}+k}(2\pi^{2})^{n-3}}\left(\frac{B_{0}+O\left(\frac{1}{n-\sum_{i=1}^{k}c_{i}+k}\right)}{B_{0}+O\left(\frac{1}{n}\right)}\right)-\left(\frac{x_{0}}{2\pi^{2}}\right)^{\sum_{i=1}^{k}\left(c_{i}-1\right)}\right|\\
= & \left(\frac{x_{0}}{2\pi^{2}}\right)^{\sum_{i=1}^{k}\left(c_{i}-1\right)}\left|\frac{n^{\frac{7}{2}}}{\left(n-\sum_{i=1}^{k}c_{i}+k\right)^{\frac{7}{2}}}\frac{\left(n-\sum_{i=1}^{k}c_{i}+k\right)!}{n^{k}\left(n-\sum_{i=1}^{k}c_{i}\right)!}\left(1+O\left(\frac{1}{n}\right)\right)-1\right|\\
= & \left(\frac{x_{0}}{2\pi^{2}}\right)^{\sum_{i=1}^{k}\left(c_{i}-1\right)}\left|\left(1+\frac{\sum_{i=1}^{k}c_{i}-k}{n-\sum_{i=1}^{k}c_{i}+k}\right)^{\frac{7}{2}}\left(1-\frac{\sum_{i=1}^{k}c_{i}-k}{n}\right)\cdot\cdots\cdot\left(1-\frac{\sum_{i=1}^{k}c_{i}-1}{n}\right)\left(1+O\left(\frac{1}{n}\right)\right)-1\right|\\
= & \left(\frac{x_{0}}{2\pi^{2}}\right)^{\sum_{i=1}^{k}\left(c_{i}-1\right)}\left|\left(1+O\left(\frac{k^{2}}{n^{\frac{1}{2}}}\right)\right)-1\right|=O\left(\frac{k^{2}}{n^{\frac{1}{2}}}\left(\frac{x_{0}}{2\pi^{2}}\right)^{\sum_{i=1}^{k}\left(c_{i}-1\right)}\right).
\end{align*}
Thus we can control the contribution of $(a)$ by 
\begin{align*}
 & \sum_{c_{1}=2}^{\lfloor n^{\frac{1}{2}}\rfloor}\cdots\sum_{c_{k}=2}^{\lfloor n^{\frac{1}{2}}\rfloor}\frac{V_{c_{1}+1}\cdot\cdots\cdot V_{c_{k}+1}}{c_{1}!\cdot\cdots\cdot c_{k}!}\left|\frac{nV_{n-\sum_{i=1}^{k}c_{i}+k}}{n^{k}V_{n}\left(n-\sum_{i=1}^{k}c_{i}\right)!}-\left(\frac{x_{0}}{2\pi^{2}}\right)^{\sum_{i=1}^{k}\left(c_{i}-1\right)}\right|\\
=O & \left(\frac{k^{2}}{n^{\frac{1}{2}}}\sum_{c_{1}=2}^{\lfloor n^{\frac{1}{2}}\rfloor}\cdots\sum_{c_{k}=2}^{\lfloor n^{\frac{1}{2}}\rfloor}\frac{V_{c_{1}+1}\cdot\cdots\cdot V_{c_{k}+1}}{c_{1}!\cdot\cdots\cdot c_{k}!}\left(\frac{x_{0}}{2\pi^{2}}\right)^{\sum_{i=1}^{k}\left(c_{i}-1\right)}\right)=O\left(\frac{k^{2}\alpha^{k}}{\sqrt{n}}\right).
\end{align*}
To bound the contribution of $(b)$ we notice that 
\begin{align}
\sum_{c_{1}=\lceil n^{\frac{1}{2}}\rceil}^{\infty}\sum_{c_{2}=2}^{\infty}\cdots\sum_{c_{k}=2}^{\infty}\left(\frac{x_{0}}{2\pi^{2}}\right)^{\sum_{i=1}^{k}\left(c_{i}-1\right)}\prod_{i=1}^{k}\frac{V_{c_{i}+1}}{c_{i}!} & =\left(\sum_{c=\lceil n^{\frac{1}{2}}\rceil}^{\infty}\frac{V_{c+1}x_{0}^{c-1}}{c!(2\pi^{2})^{c-1}}\right)\alpha^{k-1}.\label{eq:b-equality}
\end{align}
Applying volume estimates of Theorem \ref{thm:ZografManin}, we can
control the summation by 
\begin{align}
\sum_{c=\lceil n^{\frac{1}{2}}\rceil}^{\infty}\frac{V_{c+1}x_{0}^{c-1}}{c!(2\pi^{2})^{c-1}} & =\frac{x_{0}^{-2}}{2\pi^{2}}\left(B_{0}+O\left(\frac{1}{n^{\frac{1}{2}}}\right)\right)\sum_{c=\lceil n^{\frac{1}{2}}\rceil}^{\infty}\frac{c+1}{\left(c+2\right)^{\frac{7}{2}}}\leq\frac{x_{0}^{-2}}{2\pi^{2}}\left(B_{0}+O\left(\frac{1}{n^{\frac{1}{2}}}\right)\right)\frac{1}{n^{\frac{3}{4}}}=O\left(\frac{1}{n^{\frac{3}{4}}}\right),\label{eq:nsqrt-infty-sum-estimate}
\end{align}
which together with (\ref{eq:b-equality}) shows that $(b)=O\left(\frac{\alpha^{k-1}}{n^{\frac{3}{4}}}\right)$.
Finally we control the contribution of $(c)$. First we look at 
\[
\sum_{\substack{(c_{1},\ldots,c_{k})\in\mathbb{N}^{k}\\
\exists\,i\,:\,c_{j}\geq n^{\frac{1}{2}}\\
\sum_{i=1}^{k}c_{i}\leqslant\frac{n}{2}\\
c_{i}\geq2
}
}{n \choose c_{1},\ldots,c_{k}}\frac{1}{n^{k}}\frac{V_{c_{1}+1}\cdots V_{c_{k}+1}V_{n-\sum_{i=1}^{k}c_{i}+k}}{V_{n}}.
\]
If $\sum_{i=1}^{k}c_{i}\leqslant\frac{n}{2}$, then by the volume
estimates of Theorem \ref{thm:ZografManin}, we have
\begin{align*}
\frac{V_{n-\sum_{i=1}^{k}c_{i}+k}}{\left(n-\sum_{i=1}^{k}c_{i}\right)!}\frac{n!}{n^{k}V_{n}} & =\frac{\left(n-\sum_{i=1}^{k}c_{i}+k\right)!}{n^{k}\left(n-\sum_{i=1}^{k}c_{i}\right)!}\frac{(n+1)^{\frac{7}{2}}}{\left(n-\sum_{i=1}^{k}c_{i}+k+1\right)^{\frac{7}{2}}}\left(\frac{x_{0}}{2\pi^{2}}\right)^{\sum_{i=1}^{k}\left(c_{i}-1\right)}\left(1+O\left(\frac{1}{n}\right)\right)\\
 & \leq\left(\frac{x_{0}}{2\pi^{2}}\right)^{\sum_{i=1}^{k}\left(c_{i}-1\right)}\left(1+O\left(\frac{1}{n}\right)\right).
\end{align*}
Then 
\begin{align*}
 & \sum_{\substack{(c_{1},\ldots,c_{k})\in\mathbb{N}^{k}\\
\exists\,i\,:\,c_{j}\geq n^{\frac{1}{2}}\\
\sum_{i=1}^{k}c_{i}\leqslant\frac{n}{2}\\
c_{i}\geq2
}
}{n \choose c_{1},\ldots,c_{k}}\frac{1}{n^{k}}\frac{V_{c_{1}+1}\cdots V_{c_{k}+1}V_{n-\sum_{i=1}^{k}c_{i}+k}}{V_{n}}\\
 & \leq\sum_{\substack{(c_{1},\ldots,c_{k})\in\mathbb{N}^{k}\\
\exists\,i\,:\,c_{j}\geq n^{\frac{1}{2}}\\
\sum_{i=1}^{k}c_{i}\leqslant\frac{n}{2}\\
c_{i}\geq2
}
}\frac{V_{c_{1}+1}\left(\frac{x_{0}}{2\pi^{2}}\right)^{c_{1}-1}\cdot\cdots\cdot V_{c_{k}+1}\left(\frac{x_{0}}{2\pi^{2}}\right)^{c_{k}-1}}{c_{1}!\cdot\cdots\cdot c_{k}!}\left(1+O\left(\frac{1}{n}\right)\right)\\
 & \leq k\left(\sum_{c=\lceil n^{\frac{1}{2}}\rceil}^{\infty}\frac{V_{c_{1}+1}x_{0}^{c-1}}{c!(2\pi^{2})^{c-1}}\right)\alpha^{k-1}\left(1+O\left(\frac{1}{n}\right)\right)=O\left(\frac{k\alpha^{k-1}}{n^{\frac{3}{4}}}\right),
\end{align*}
where the last bound follows from (\ref{eq:nsqrt-infty-sum-estimate}).
We now consider the contribution from when $\sum_{i=1}^{k}c_{i}\geq\frac{n}{2}$,
that is, of the term
\begin{equation}
\sum_{\substack{(c_{1},\ldots,c_{k})\in\mathbb{N}^{k}\\
\frac{n}{2}\leqslant\sum_{i=1}^{k}c_{i}\leqslant n\\
c_{i}\geq2
}
}{n \choose c_{1},\ldots,c_{k}}\frac{1}{n^{k}}\frac{V_{c_{1}+1}\cdots V_{c_{k}+1}V_{n-\sum_{i=1}^{k}c_{i}+k}}{V_{n}}.\label{eq:eq-c-for-large-sum}
\end{equation}
Since $\sum_{i=1}^{k}c_{i}\geqslant\frac{n}{2}$, there exists some
$\ell$ such that $c_{\ell}\geqslant\frac{n}{2k}$, and thus it follows
that 
\begin{align*}
 & \sum_{\substack{(c_{1},\ldots,c_{k})\in\mathbb{N}^{k}\\
\frac{n}{2}\leqslant\sum_{i=1}^{k}c_{i}\leqslant n\\
c_{i}\geq2
}
}\frac{V_{c_{1}+1}\cdot\cdots\cdot V_{c_{k}+1}\cdot V_{n-\sum_{i=1}^{k}c_{i}+k}}{c_{1}!\cdot\cdots\cdot c_{k}!\left(n-\sum_{i=1}^{k}c_{i}\right)!}\frac{n!}{n^{k}V_{n}}\\
 & \leq k\sum_{\substack{(c_{1},\ldots,c_{k})\in\mathbb{N}^{k}\\
\frac{n}{2}\leqslant\sum_{i=1}^{k}c_{i}\leqslant n\\
c_{1}\geq\left\lfloor \frac{n}{2k}\right\rfloor \\
c_{i}\geq2
}
}\frac{V_{c_{1}+1}\cdot\cdots\cdot V_{c_{k}+1}\cdot V_{n-\sum_{i=1}^{k}c_{i}+k}}{c_{1}!\cdot\cdots\cdot c_{k}!\left(n-\sum_{i=1}^{k}c_{i}\right)!}\frac{n!}{n^{k}V_{n}}.\\
 & =k\sum_{\substack{(c_{1},\ldots,c_{k},c_{k+1})\in\mathbb{N}^{k+1}\\
\sum_{i=1}^{k+1}c_{i}=n\\
c_{1}\geq\left\lfloor \frac{n}{2k}\right\rfloor \\
0\leq c_{k+1}\leq\frac{n}{2}\\
c_{i}\geq2,\ i<k+1
}
}\frac{V_{c_{1}+1}\cdot\cdots\cdot V_{c_{k}+1}\cdot V_{c_{k+1}+k}}{c_{1}!\cdot\cdots\cdot c_{k}!c_{k+1}!}\frac{n!}{n^{k}V_{n}}
\end{align*}
By the volume estimate of Theorem \ref{thm:ZografManin} we then obtain
\[
\frac{V_{c_{1}+1}}{c_{1}!}\frac{n!}{n^{k}V_{n}}=\frac{(n+1)^{\frac{7}{2}}}{(c_{1}+2)^{\frac{5}{2}}}\frac{1}{n^{k}}\left(\frac{x_{0}}{2\pi^{2}}\right)^{n-c_{1}-1}\left(1+O\left(\frac{k}{n}\right)\right)=\frac{1}{n^{k-1}}\left(\frac{x_{0}}{2\pi^{2}}\right)^{\sum_{i=2}^{k+1}c_{i}-1}\left(1+O\left(\frac{k}{n}\right)\right).
\]
Moreover, by Lemma \ref{lem:mirzakhani-vol-comparison}, there exists
a constant $\gamma>0$ such that 
\[
V_{c_{k+1}+k}\leqslant\gamma^{k-3}\left(c_{k+1}+k\right)^{k-3}V_{c_{k+1}+3}.
\]
Hence,
\begin{align*}
 & k\sum_{\substack{(c_{1},\ldots,c_{k},c_{k+1})\in\mathbb{N}^{k+1}\\
\sum_{i=1}^{k+1}c_{i}=n\\
c_{1}\geq\left\lfloor \frac{n}{2k}\right\rfloor \\
0\leq c_{k+1}\leq\frac{n}{2}\\
c_{i}\geq2,\ i<k+1
}
}\frac{V_{c_{1}+1}\cdot\cdots\cdot V_{c_{k}+1}\cdot V_{c_{k+1}+k}}{c_{1}!\cdot\cdots\cdot c_{k}!c_{k+1}!}\frac{n!}{n^{k}V_{n}}\\
 & \leq k\sum_{\substack{(c_{2},\ldots,c_{k},c_{k+1})\in\mathbb{N}^{k}\\
\sum_{i=2}^{k+1}c_{i}\leq n-\left\lfloor \frac{n}{2k}\right\rfloor \\
0\leq c_{k+1}\leq\frac{n}{2}\\
c_{i}\geq2,\ i<k+1
}
}\frac{V_{c_{2}+1}\cdots V_{c_{k+1}+k}}{c_{2}!\cdots c_{k+1}!}\left(\frac{x_{0}}{2\pi^{2}}\right)^{\sum_{i=2}^{k+1}c_{i}-1}\frac{1}{n^{k-1}}\left(1+O\left(\frac{k}{n}\right)\right)\\
 & \leq\gamma^{k-3}k\sum_{\substack{(c_{2},\ldots,c_{k},c_{k+1})\in\mathbb{N}^{k}\\
\sum_{i=2}^{k+1}c_{i}\leq n-\left\lfloor \frac{n}{2k}\right\rfloor \\
0\leq c_{k+1}\leq\frac{n}{2}\\
c_{i}\geq2,\ i<k+1
}
}\frac{V_{c_{2}+1}\cdots V_{c_{k+1}+3}}{c_{2}!\cdots(c_{k+1}+1)!}\left(\frac{x_{0}}{2\pi^{2}}\right)^{\sum_{i=2}^{k+1}c_{i}-1}\frac{(c_{k+1}+k)^{k-2}}{n^{k-1}}\left(1+O\left(\frac{k}{n}\right)\right)
\end{align*}
Since $c_{k+1}\leq\frac{n}{2}$, for $n$ sufficiently large (i.e.
for $k\leq\frac{n}{2}$) we have $\frac{(c_{k+1}+k)^{k-2}}{n^{k-1}}\leq\frac{1}{n}$
and thus
\begin{align*}
 & \sum_{\substack{(c_{1},\ldots,c_{k})\in\mathbb{N}^{k}\\
\frac{n}{2}\leqslant\sum_{i=1}^{k}c_{i}\leqslant n\\
c_{i}\geq2
}
}\frac{V_{c_{1}+1}\cdot\cdots\cdot V_{c_{k}+1}\cdot V_{n-\sum_{i=1}^{k}c_{i}+k}}{c_{1}!\cdot\cdots\cdot c_{k}!\left(n-\sum_{i=1}^{k}c_{i}\right)!}\frac{n!}{n^{k}V_{n}}\\
 & \leq\frac{\gamma^{k-3}k}{n}\sum_{\substack{(c_{2},\ldots,c_{k},c_{k+1})\in\mathbb{N}^{k}\\
\sum_{i=2}^{k+1}c_{i}\leq n-\left\lfloor \frac{n}{2k}\right\rfloor \\
0\leq c_{k+1}\leq\frac{n}{2}\\
c_{i}\geq2,\ i<k+1
}
}\frac{V_{c_{2}+1}\cdots V_{c_{k+1}+3}}{c_{2}!\cdots(c_{k+1}+1)!}\left(\frac{x_{0}}{2\pi^{2}}\right)^{\sum_{i=2}^{k+1}c_{i}-1}\left(1+O\left(\frac{k}{n}\right)\right)\\
 & \leq\frac{\gamma^{k-3}k}{n}\left(\sum_{c_{k+1}=0}^{\left\lfloor \frac{n}{2}\right\rfloor }\frac{V_{c_{k+1}+3}}{(c_{k+1}+1)!}\left(\frac{x_{0}}{2\pi^{2}}\right)^{c_{k+1}-1}\right)\left(\sum_{c_{2}=2}^{\infty}\cdots\sum_{c_{k}=2}^{\infty}\left(\frac{x_{0}}{2\pi^{2}}\right)^{\sum_{i=2}^{k}c_{i}}\prod_{i=2}^{k}\frac{V_{c_{i}+1}}{c_{i}!}\right)\left(1+O\left(\frac{k}{n}\right)\right)\\
 & \leq\frac{\gamma^{k-3}k\alpha^{k-1}}{n}\left(\frac{x_{0}}{2\pi^{2}}\right)^{k-1}\left(\sum_{c_{k+1}=0}^{\infty}\frac{V_{c_{k+1}+3}}{(c_{k+1}+1)!}\left(\frac{x_{0}}{2\pi^{2}}\right)^{c_{k+1}-1}\right)\left(1+O\left(\frac{k}{n}\right)\right)
\end{align*}
This latter summation converges since for sufficiently large $c_{k+1}$,
we have the asymptotic from Theorem \ref{thm:ZografManin} asserting
that 
\[
\frac{V_{c_{k+1}+3}}{(c_{k+1}+1)!}\left(\frac{x_{0}}{2\pi^{2}}\right)^{c_{k+1}-1}=2\pi^{2}x_{0}^{-4}\frac{1}{c_{k+1}^{\frac{3}{2}}}\left(1+O\left(\frac{1}{c_{k+1}}\right)\right).
\]
In summary, we have that

\begin{align*}
(a) & =O\left(\frac{k^{2}\alpha^{k}}{\sqrt{n}}\right),\\
(b) & =O\left(\frac{\alpha^{k-1}}{n^{\frac{3}{4}}}\right),\\
(c) & =O\left(\frac{\gamma^{k-3}k\alpha^{k-1}}{n}+\frac{k\alpha^{k-1}}{n^{\frac{3}{4}}}\right),
\end{align*}
and hence the result follows for $k\geq3$.

For $k=1$, by Lemma \ref{lem:unnested-exp-ub} it is sufficient to
show
\[
\sum_{c=2}^{\left\lfloor \frac{n}{2}\right\rfloor }{n \choose c}\frac{1}{n}\frac{V_{c+1}V_{n-c+1}}{V_{n}}=\alpha+O\left(\frac{1}{\sqrt{n}}\right).
\]
As before we observe that 
\begin{align*}
 & \left|\sum_{c=2}^{\left\lfloor \frac{n}{2}\right\rfloor }{n \choose c}\frac{1}{n}\frac{V_{c+1}V_{n-c+1}}{V_{n}}-\alpha\right|\\
 & \leq\underbrace{\left|\sum_{c=2}^{\left\lfloor \sqrt{n}\right\rfloor }{n \choose c}\frac{1}{n}\frac{V_{c+1}V_{n-c+1}}{V_{n}}-\frac{V_{c+1}}{c!}\left(\frac{x_{0}}{2\pi^{2}}\right)^{c-1}\right|}_{(a)}+\underbrace{\sum_{c=\left\lfloor \sqrt{n}\right\rfloor +1}^{\infty}\frac{V_{c+1}}{c!}\left(\frac{x_{0}}{2\pi^{2}}\right)^{c-1}}_{(b)}+\underbrace{\sum_{c=\left\lfloor \sqrt{n}\right\rfloor +1}^{\left\lfloor \frac{n}{2}\right\rfloor }\frac{1}{n}\frac{V_{c+1}V_{n-c+1}n!}{V_{n}(n-c)!c!}}_{(c)}.
\end{align*}
Identically to the analysis of the terms (a) and (b) for $k\geq3$,
we obtain 
\begin{align*}
(a) & =O\left(\frac{1}{\sqrt{n}}\right),\\
(b) & =O\left(\frac{1}{n^{\frac{3}{4}}}\right).
\end{align*}
The bound for (c) is simpler than before since because $c\leq\frac{n}{2}$,
it follows that $n-c+1\geq\frac{n}{2}$ and so 
\[
\frac{n!V_{n-c+1}}{(n-c)!nV_{n}}=\left(\frac{x_{0}}{2\pi^{2}}\right)^{c-1}\left(1+O\left(\frac{1}{n}\right)\right).
\]
Thus by Theorem \ref{thm:ZografManin} 
\[
\sum_{c=\left\lfloor \sqrt{n}\right\rfloor +1}^{\left\lfloor \frac{n}{2}\right\rfloor }\frac{1}{n}\frac{V_{c+1}V_{n-c+1}n!}{V_{n}(n-c)!c!}=\sum_{c=\left\lfloor \sqrt{n}\right\rfloor +1}^{\left\lfloor \frac{n}{2}\right\rfloor }\frac{V_{c+1}}{c!}\left(\frac{x_{0}}{2\pi^{2}}\right)^{c-1}\left(1+O\left(\frac{1}{n}\right)\right)=O\left(\frac{1}{n^{\frac{3}{4}}}\right).
\]
This leaves the case for $k=2$, where again by Lemma \ref{lem:unnested-exp-ub},
it is sufficient to show

\[
\sum_{\substack{\substack{(c_{1},}
c_{2})\in\mathbb{N}^{2},\\
c_{1}+c_{2}\leq n-1,\\
c_{i}\geq2
}
}{n \choose c_{1},c_{2}}\frac{1}{n^{2}}\frac{V_{c_{1}+1}V_{c_{2}+1}V_{n-c_{1}-c_{2}+2}}{V_{n}}=\alpha^{2}+O\left(\frac{1}{\sqrt{n}}\right).
\]
We observe that 
\begin{align*}
 & \left|\sum_{\substack{\substack{(c_{1},}
c_{2})\in\mathbb{N}^{2},\\
c_{1}+c_{2}\leq n-1,\\
c_{i}\geq2
}
}{n \choose c_{1},c_{2}}\frac{1}{n^{2}}\frac{V_{c_{1}+1}V_{c_{2}+1}V_{n-c_{1}-c_{2}+2}}{V_{n}}-\alpha^{2}\right|\\
 & \leq\underbrace{\left|\sum_{c_{1}=2}^{\left\lfloor \sqrt{n}\right\rfloor }\sum_{c_{2}=2}^{\left\lfloor \sqrt{n}\right\rfloor }{n \choose c_{1},c_{2}}\frac{1}{n^{2}}\frac{V_{c_{1}+1}V_{c_{2}+1}V_{n-c_{1}-c_{2}+2}}{V_{n}}-\left(\frac{x_{0}}{2\pi^{2}}\right)^{c_{1}-1+c_{2}-1}\frac{V_{c_{1}+1}V_{c_{2}+1}}{c_{1}!c_{2}!}\right|}_{(a)}\\
 & +2\underbrace{\sum_{c_{1}=\left\lfloor \sqrt{n}\right\rfloor }^{\infty}\sum_{c_{2}=2}^{\infty}\left(\frac{x_{0}}{2\pi^{2}}\right)^{c_{1}-1+c_{2}-1}\frac{V_{c_{1}+1}V_{c_{2}+1}}{c_{1}!c_{2}!}}_{(b)}+\underbrace{\sum_{\substack{\substack{(c_{1},}
c_{2})\in\mathbb{N}^{2},\\
c_{1}+c_{2}\leq n-1,\\
c_{i}\geq2\\
\exists\,i\,:c_{i}\geq\sqrt{n}
}
}{n \choose c_{1},c_{2}}\frac{1}{n^{2}}\frac{V_{c_{1}+1}V_{c_{2}+1}V_{n-c_{1}-c_{2}+2}}{V_{n}}}_{(c)}.
\end{align*}

The bounds on each of these terms follow very similarly to the case
when $k\geq3$. The only difference is for (c), where now the analogue
of $c_{3}$ takes values between $1$ and $\left\lfloor \frac{n}{2}\right\rfloor $,
and there is no need to use the Mirzakhani volume bounds from Lemma
\ref{lem:mirzakhani-vol-comparison} on the $V_{c_{3}+2}$ term.
\end{proof}
We now conclude the proof of Theorem \ref{thm:main-thm} for $\ell=1$.
\begin{proof}[Proof of Theorem \ref{thm:main-thm} for $\ell=1$]
 Combining (\ref{eq:nested + unnested}), Proposition \ref{prop:Nested contribution}
and Proposition \ref{prop:Unnested contribution}, we see that for
fixed $k$,
\[
\mathbb{E}_{n}\left(\left(N_{n,[a,b]}(X)\right)_{k}\right)=\left(\frac{b^{2}-a^{2}}{2}\right)^{k}\alpha^{k}\left(1+O_{k}\left(\frac{1}{n^{\frac{1}{2}}}\right)\right)+O_{k}\left(\frac{b^{2k}}{n}\right)\to\left(\frac{b^{2}-a^{2}}{2}\right)^{k}\alpha^{k},
\]
as $n\to\infty$ and so Theorem \ref{thm:main-thm} follows from Proposition
\ref{prop:factorial-moments} with the mean
\[
\frac{b^{2}-a^{2}}{2}\alpha.
\]
To see that $\alpha=\sum_{i=2}^{\infty}\frac{V_{i+1}}{i!}\left(\frac{x_{0}}{2\pi^{2}}\right)^{i-1}=\frac{\left(j_{0}\pi\right)^{2}}{4}\left(1-\frac{J_{3}(j_{0})}{J_{1}(j_{0})}\right)$,
consider the function 

\[
\varphi_{0}(x)=\sum_{i=3}^{\infty}\frac{V_{i}}{i!}\frac{x^{i}}{(2\pi^{2})^{i-3}},
\]
which by Manin and Zograf \cite[Proof of Theorem 6.1]{Ma.Zo00} has
radius of convergence given by $x_{0}$. By differentiating, we see
that $\alpha=\frac{2\pi^{2}\varphi_{0}'(x_{0})}{x_{0}}$. Define
$y(x)=\varphi_{0}''(x)$, then from \cite[Proof of Theorem 6.1]{Ma.Zo00}
and the introduction of \cite{Ka.Ma.Za96}, $y(x)$ can be obtained
by inverting $x(y)=-\sqrt{y}J_{0}'(2\sqrt{y})$ .

By direct computation and using the recurrence relations satisfied
by Bessel functions, we see that $x'(y)=J_{0}(2\sqrt{y})$. Thus on
the domain $\left[0,\frac{j_{0}^{2}}{4}\right]$, this derivative
is monotonically decreasing from 1 to 0, since $j_{0}$ is the first
positive zero of $J_{0}$. It follows that $x(y)$ is monotonically
increasing on the same domain from $x(0)=0$ to $x\left(\frac{j_{0}^{2}}{4}\right)=x_{0}$.
We wish to compute $\varphi'_{0}(x_{0})$ which, since $\varphi'_{0}(0)=0$,
is given by

\[
\varphi_{0}'(x_{0})=\int_{0}^{x_{0}}y(x)\mathrm{d}x.
\]
Using the definition of $x$ and its properties demonstrated above,
the area under the curve $y(x)$ in the domain $[0,x_{0}]=\left[x(0),x\left(\frac{j_{0}^{2}}{4}\right)\right]$
is precisely $x_{0}\frac{j_{0}^{2}}{4}$ minus the area under the
curve $x(y)$ in the domain $\left[0,\frac{j_{0}^{2}}{4}\right]$.
Thus,

\begin{align*}
\varphi_{0}'(x_{0}) & =x_{0}\frac{j_{0}^{2}}{4}-\int_{0}^{\frac{j_{0}^{2}}{4}}x(y)\mathrm{d}y\\
 & =\int_{0}^{\frac{j_{0}^{2}}{4}}yJ_{0}(2\sqrt{y})\mathrm{d}y\\
 & =\frac{1}{8}\int_{0}^{j_{0}}y^{3}J_{0}(y)\mathrm{d}y.
\end{align*}
To evaluate this latter integral, we recall the identity (see \cite[Equation 5.52]{Gr.Ry07})

\[
\int x^{p+1}J_{p}(x)\mathrm{d}x=x^{p+1}J_{p+1}(x),
\]
and so integration by parts gives 
\begin{align}
\varphi_{0}'(x_{0}) & =\frac{1}{8}\left(j_{0}^{3}J_{1}(j_{0})-2\int_{0}^{j_{0}}y^{2}J_{1}(y)\mathrm{d}y\right)\nonumber \\
 & =\frac{1}{8}\left(j_{0}^{3}J_{1}(j_{0})-2j_{0}^{2}J_{2}(j_{0})\right)\nonumber \\
 & =\frac{j_{0}^{3}}{16}\left(J_{1}(j_{0})-J_{3}(j_{0})\right),\label{eq:bessel-function}
\end{align}
where in the last step we use the Bessel function identity $\frac{4}{x}J_{2}(x)=J_{1}(x)+J_{3}(x)$.
Since $J_{0}'(x)=-J_{1}(x),$ we also have $x_{0}=-\frac{j_{0}}{2}J_{0}'(j_{0})=\frac{j_{0}}{2}J_{1}(j_{0})$,
and so from (\ref{eq:bessel-function}) we see that
\[
\alpha=\frac{2\pi^{2}\varphi_{0}'(x_{0})}{x_{0}}=\frac{(j_{0}\pi)^{2}}{4}\left(1-\frac{J_{3}(j_{0})}{J_{1}(j_{0})}\right).
\]
\end{proof}

\subsection{Proof of Theorem \ref{thm:main-thm}}

\label{subsec:main-thm-proof}

We now prove Theorem \ref{thm:main-thm} using the results on nested
and unnested multicurves from before. Let $\ell\in\mathbb{N}$ and
suppose that $0\leq a_{i}<b_{i}$ are real numbers for $i=1,\ldots,\ell$
such that the intervals $[a_{i},b_{i}]$ are pairwise disjoint. For
sufficiently large $n$, we have $\frac{b_{i}}{\sqrt{n}}<2\mathrm{arcsinh}(1)$
for each $i=1,\ldots,\ell$ and so given $k_{1},\ldots,k_{\ell}\in\mathbb{N}$
and $X\in\mathcal{M}_{0,n}$, the product
\[
\left(N_{n,[a_{1},b_{1}]}(X)\right)_{k_{1}}\cdots\left(N_{n,[a_{\ell},b_{\ell}]}(X)\right)_{k_{\ell}}
\]
counts the number of ordered $\ell$ tuples whose $i^{\mathrm{th}}$
entry is an ordered $k_{i}$ tuple consisting of distinct, disjoint,
primitive simple closed geodesics on $X$, whose lengths are in the
interval $\left[\frac{a_{i}}{\sqrt{n}},\frac{b_{i}}{\sqrt{n}}\right]$.
Since the intervals $[a_{i},b_{i}]$ are pairwise disjoint, the geodesics
in the $i^{\mathrm{th}}$ and $j^{\mathrm{th}}$ tuples are distinct.
This means, the product of the factorials counts the number of ordered
$\sum_{i=1}^{\ell}k_{i}$ multigeodesics on $X$ such that the $i^{\mathrm{th}}$
block (of length $k_{i}$) of curves have lengths in $\left[\frac{a_{i}}{\sqrt{n}},\frac{b_{i}}{\sqrt{n}}\right]$. 

As before, we can separate the different mapping class group orbits
of these multicurves into nested and unnested curves. Moreover, we
can use identical topological descriptions of these multicurve types
as in subsections \ref{subsec:nested} and \ref{subsec:unnested}
to see that with an application of Mirzakhani's integration formula
Theorem\ref{thm:MIF}, the contribution of the nested multicurves
to the expectation is then bounded by
\begin{equation}
\sum_{\left\{ (c_{j},d_{j})\right\} _{j=1}^{j=1+\tilde{k}}\in\mathcal{A}_{\tilde{k}}^{o}}\tilde{k}!{n \choose c_{1,}\ldots,c_{\tilde{k}}}\frac{\tilde{b}^{2\tilde{k}}}{2^{\sum_{i=1}^{\ell}k_{i}}}\frac{\prod_{j=1}^{1+\tilde{k}}V_{c_{j}+d_{j}}}{n^{\tilde{k}}V_{n}}\left(1+O\left(\frac{\tilde{k}\tilde{b}^{2}}{n}\right)\right).\label{eq:full-nested-count}
\end{equation}
where $\tilde{k}=\sum_{i=1}^{\ell}k_{i}$ and $\tilde{b}=\max_{j=1,\ldots,\ell}b_{i}$.
Applying Proposition \ref{prop:Nested contribution} with $\tilde{k}$
and $\tilde{b}$, we see that this is $O\left(\frac{1}{n}\right)$.
In a similar vein, we can identically consider the contribution of
the unnested multicurves to the expectation, and see it is equal to 

\[
\sum_{\substack{(c_{1},\ldots,c_{\tilde{k}})\in\mathbb{N}^{\tilde{k}},\\
\sum_{i=1}^{\tilde{k}}c_{i}\leq n,\\
c_{i}\geq2.
}
}{n \choose c_{1},\ldots,c_{\tilde{k}}}\frac{1}{n^{\tilde{k}}}\frac{V_{c_{1}+1}\cdots V_{\tilde{k}+1}V_{n-\sum_{i=1}^{\tilde{k}}c_{i}+\tilde{k}}}{V_{n}}\prod_{j=1}^{\ell}\left(\frac{b_{j}^{2}-a_{j}^{2}}{2}\right)^{k_{j}}\left(1+O\left(\frac{\tilde{k}\tilde{b}^{2}}{n}\right)\right),
\]
again, where $\tilde{k}=\sum_{i=1}^{\ell}k_{i}$. By Proposition \ref{prop:Unnested contribution},
this is equal to 

\[
\left(\prod_{i=1}^{\ell}\left(\alpha\frac{b_{i}^{2}-a_{i}^{2}}{2}\right)^{k_{i}}\right)\left(1+O_{\tilde{k}}\left(\frac{1}{\sqrt{n}}\right)\right),
\]
where 

\[
\alpha=\sum_{i=2}^{\infty}\frac{V_{i+1}}{i!}\left(\frac{x_{0}}{2\pi^{2}}\right)^{i-1}=\frac{\left(j_{0}\pi\right)^{2}}{4}\left(1-\frac{J_{3}(j_{0})}{J_{1}(j_{0})}\right).
\]
Combining these contributions to evaluate $\mathbb{E}_{n}\left(\left(N_{n,[a_{1},b_{1}]}(X)\right)_{k_{1}}\cdots\left(N_{n,[a_{\ell},b_{\ell}]}(X)\right)_{k_{\ell}}\right)$
and taking $n\to\infty$ we obtain Theorem \ref{thm:main-thm} by
an application of Proposition \ref{prop:factorial-moments}.

\subsection{Proof of Theorem \ref{thm:main-thm-2}}

\label{subsec:main-thm2-proof}

The proof of Theorem \ref{thm:main-thm-2} is similar to Theorem \ref{thm:main-thm}
except that we are only required to examine a single topological type
of multicurves. Recall that for any integer $c\geq2$, any $n\in\mathbb{N}$
and any real numbers $0\leq a<b$, we defined $N_{n,c,[a,b]}(X)$
to be the number of primitive closed geodesics on $X$ that separate
off exactly $c$ cusps from $X$ with length in the interval $\left[\frac{a}{\sqrt{n}},\frac{b}{\sqrt{n}}\right]$. 

We will use the method of factorial moments to prove Theorem \ref{thm:main-thm-2},
and so for $k_{1},\ldots,k_{\ell}\in\mathbb{N}$, we are required
to compute the expected value of 
\begin{equation}
\left(N_{n,c_{1},[a_{1},b_{1}]}(X)\right)_{k_{1}}\cdots\left(N_{n,c_{\ell},[a_{\ell},b_{\ell}]}(X)\right)_{k_{\ell}},\label{eq:product-factorials}
\end{equation}
where $c_{1},\ldots,c_{\ell}\geq2$ are distinct integers and $0\leq a_{i}<b_{i}$
are real numbers for $i=1,\ldots,\ell$.

Notice that this product counts the number of ordered $\ell$ tuples
whose $i^{\mathrm{th}}$ entry is an ordered $k_{i}$ tuple of distinct
primitive closed geodesics on $X$ that separate off exactly $c_{i}$
cusps from $X$, and whose lengths are in the interval $\left[\frac{a_{i}}{\sqrt{n}},\frac{b_{i}}{\sqrt{n}}\right]$.
Since the integers $c_{i}$ are distinct, the geodesics in the $i^{\mathrm{th}}$
entry of this tuple are distinct from those in the $j^{\mathrm{th}}$
tuple for any $i\neq j$. Moreover, for $n$ sufficiently large, $\frac{b_{i}}{\sqrt{n}}<2\mathrm{arcsinh}(1)$
for each $i=1,\ldots,\ell$, and so the geodesics considered are simple
and disjoint from one another. This means that the product counts
the number of ordered multigeodesics of length $\sum_{i=1}^{k}k_{i}$
on $X$ such that geodesics in the $i^{\mathrm{th}}$ block separate
off exactly $c_{i}$ cusps, and have lengths in $\left[\frac{a_{i}}{\sqrt{n}},\frac{b_{i}}{\sqrt{n}}\right]$. 

Now, we split this count into counts of nested and unnested multicurves
of this topological type. The number of nested multicurves is bounded
by the number of nested multicurves considered in the proof of Theorem
\ref{thm:main-thm} since we are considering only a subset of the
possible topological types that can occur due to the limitations on
how many cusps each multicurve component separates off. It follows
that the expected number of nested multicurves included in the count
here is bounded by (\ref{eq:full-nested-count}) which we saw to be
$O_{\tilde{k}}\left(\frac{\tilde{b}^{2\tilde{k}}}{n}\right)$ where
$\tilde{k}=\sum_{i=1}^{\ell}k_{i}$ and $\tilde{b}=\max_{i=1,\ldots,\ell}b_{i}$. 

Once again, the dominant contribution to the expectation will arise
from the unnested multicurves. Since we know precisely how many cusps
are separated off by each multicurve component, there is precisely
one topological type of mapping class group orbit of unnested multicurves
up to the labeling of the cusps that is considered. Indeed, a mapping
class group orbit of an unnested multicurve that is counted by (\ref{eq:product-factorials})
corresponds to the following decomposition of a puncture labeled topological
surface $\Sigma_{0,n,0}$ 

\[
\Sigma_{0,n-\sum_{i=1}^{\ell}k_{i}c_{i},\sum_{i=1}^{\ell}k_{i}}\sqcup\bigsqcup_{i=1}^{\ell}\bigsqcup_{j=1}^{k_{i}}\Sigma_{0,c_{i},1},
\]
with labels on the punctures of each subsurface inherited from labels
on $\Sigma_{0,n,0}$. The number of ways that the labels can be inherited
on the subsurfaces is given precisely by the multinomial coefficient

\[
C\eqdf{n \choose \underbrace{c_{1},\ldots,c_{1}}_{k_{1}\text{ times}},\ldots,\underbrace{c_{\ell},\ldots,c_{\ell}}_{k_{\ell}\text{ times}}}.
\]
Note that if $\sum_{i=1}^{\ell}k_{i}c_{i}>n$ then there are no such
unnested multigeodesics since the cusps separated by each curve component
are distinct, and so we assume that this summation is bounded by $n$
from now on. 

Using Mirzakhani's integration formula Theorem \ref{thm:MIF}, we
are led to compute

\begin{align*}
 & \frac{C}{V_{n}}\int_{\left[\frac{a_{1}}{\sqrt{n}},\frac{b_{1}}{\sqrt{n}}\right]^{k_{1}}}\cdots\int_{\left[\frac{a_{\ell}}{\sqrt{n}},\frac{b_{\ell}}{\sqrt{n}}\right]^{k_{\ell}}}\left(\prod_{i=1}^{\ell}\prod_{j=1}^{k_{i}}t_{i,j}V_{c_{i}+1}(\mathbf{0}_{c_{i}},t_{i,j})\right)V_{n-\sum_{i=1}^{\ell}k_{i}(c_{i}-1)}(\mathbf{0}_{n-\sum_{i=1}^{\ell}k_{i}c_{i}},\mathbf{t})\bigwedge_{i=1}^{\ell}\bigwedge_{j=1}^{k_{i}}\mathrm{d}t_{i,j}
\end{align*}
where $\mathbf{t}=(t_{1,1},\ldots,t_{1,k_{1}},\ldots,t_{\ell,1},\ldots,t_{\ell,k_{\ell}})$.
We have the following estimate.
\begin{prop}
\label{prop:unnested-contribution-fixecusps}For distinct integers
$c_{1},\ldots,c_{\ell}\geq2$ and $k_{1}\ldots,k_{\ell}\in\mathbb{N}$
satisfying $\sum_{i=1}^{\ell}k_{i}c_{i}\leq n$, we have that as $n\to\infty$,

\begin{align*}
 & \frac{C}{V_{n}}\int_{\left[\frac{a_{1}}{\sqrt{n}},\frac{b_{1}}{\sqrt{n}}\right]^{k_{1}}}\cdots\int_{\left[\frac{a_{\ell}}{\sqrt{n}},\frac{b_{\ell}}{\sqrt{n}}\right]^{k_{\ell}}}\left(\prod_{i=1}^{\ell}\prod_{j=1}^{k_{i}}t_{i,j}V_{c_{i}+1}(\mathbf{0}_{c_{i}},t_{i,j})\right)V_{n-\sum_{i=1}^{\ell}k_{i}(c_{i}-1)}\left(\mathbf{0}_{n-\sum_{i=1}^{\ell}k_{i}c_{i}},\mathbf{t}\right)\bigwedge_{i=1}^{\ell}\bigwedge_{j=1}^{k_{i}}\mathrm{d}t_{i,j}\\
 & =\left(\prod_{i=1}^{\ell}\left(\frac{b_{i}^{2}-a_{i}^{2}}{2}\frac{V_{c_{i}+1}}{c_{i}!}\left(\frac{x_{0}}{2\pi^{2}}\right)^{c_{i}-1}\right)^{k_{i}}\right)\left(1+O_{\tilde{k},\tilde{b},\tilde{c},\ell}\left(\frac{1}{n}\right)\right),
\end{align*}
where $\tilde{k}=\sum_{i=1}^{\ell}k_{i}$, \textup{$\tilde{c}=\max(c_{i})$}
and $\tilde{b}=\max(b_{i})$.
\end{prop}

\begin{proof}
The proof is similar but simpler than the proof of Proposition \ref{prop:Unnested contribution}.
For each of the volumes in the integrand, we use Lemma \ref{lem:volbds}
to pass to a $\sinh$ approximation with error term and then consider
a Taylor expansion of the resulting integrand to obtain
\begin{align*}
 & \frac{C}{V_{n}}\int_{\left[\frac{a_{1}}{\sqrt{n}},\frac{b_{1}}{\sqrt{n}}\right]^{k_{1}}}\cdots\int_{\left[\frac{a_{\ell}}{\sqrt{n}},\frac{b_{\ell}}{\sqrt{n}}\right]^{k_{\ell}}}\left(\prod_{i=1}^{\ell}\prod_{j=1}^{k_{i}}t_{i,j}V_{c_{i}+1}(\mathbf{0}_{c_{i}},t_{i,j})\right)V_{n-\sum_{i=1}^{\ell}k_{i}(c_{i}-1)}(\mathbf{0}_{n-\sum_{i=1}^{\ell}k_{i}c_{i}},\mathbf{t})\bigwedge_{i=1}^{\ell}\bigwedge_{j=1}^{k_{i}}\mathrm{d}t_{i,j}\\
 & =\frac{C\left(\prod_{i=1}^{\ell}V_{c_{i}+1}^{k_{i}}\right)V_{n-\sum_{i=1}^{\ell}k_{i}(c_{i}-1)}}{V_{n}}\int_{\left[\frac{a_{1}}{\sqrt{n}},\frac{b_{1}}{\sqrt{n}}\right]^{k_{1}}}\cdots\int_{\left[\frac{a_{\ell}}{\sqrt{n}},\frac{b_{\ell}}{\sqrt{n}}\right]^{k_{\ell}}}\left(\prod_{i=1}^{\ell}\prod_{j=1}^{r_{i}}t_{i,j}\right)\bigwedge_{i=1}^{\ell}\bigwedge_{j=1}^{k_{i}}\mathrm{d}t_{i,j}\left(1+O\left(\frac{\tilde{k}\tilde{b}^{2}}{n}\right)\right)\\
 & =\frac{C\left(\prod_{i=1}^{\ell}V_{c_{i}+1}^{k_{i}}\right)V_{n-\sum_{i=1}^{\ell}k_{i}(c_{i}-1)}}{n^{\sum_{i=1}^{\ell}k_{i}}V_{n}}\prod_{i=1}^{\ell}\left(\frac{b_{i}^{2}-a_{i}^{2}}{2}\right)^{k_{i}}\left(1+O\left(\frac{\tilde{k}\tilde{b}^{2}}{n}\right)\right).
\end{align*}
By Theorem \ref{thm:ZografManin}, since the $c_{i}$ and $k_{i}$
are fixed, we have that as $n\to\infty$, 
\begin{align*}
\frac{CV_{n-\sum_{i=1}^{\ell}k_{i}(c_{i}-1)}}{n^{\sum_{i=1}^{\ell}k_{i}}V_{n}}= & \frac{(n-\sum_{i=1}^{\ell}k_{i}(c_{i}-1))!}{n^{\sum_{i=1}^{\ell}k_{i}}(n-\sum_{i=1}^{\ell}k_{i}c_{i})!}\left(\frac{n+1}{n-\sum_{i=1}^{\ell}k_{i}(c_{i}-1)+1}\right)^{\frac{7}{2}}\\
 & \cdot\frac{x_{0}^{\sum_{i=1}^{\ell}k_{i}(c_{i}-1)}}{(c_{1}!)^{k_{1}}\cdots(c_{k}!)^{k_{\ell}}(2\pi^{2})^{\sum_{i=1}^{\ell}k_{i}(c_{i}-1)}}\left(1+O\left(\frac{1}{n}\right)\right).
\end{align*}
Then we note that 
\[
\frac{(n-\sum_{i=1}^{\ell}k_{i}(c_{i}-1))!}{n^{\sum_{i=1}^{\ell}k_{i}}(n-\sum_{i=1}^{\ell}k_{i}c_{i})!}=\left(1-\frac{\sum_{i=1}^{\ell}k_{i}c_{i}-\sum_{i=1}^{\ell}c_{i}}{n}\right)\cdots\left(1-\frac{\sum_{i=1}^{\ell}k_{i}c_{i}-1}{n}\right)=1+O\left(\frac{\ell\tilde{k}\tilde{c}^{2}}{n}\right),
\]
and
\[
\left(\frac{n+1}{n-\sum_{i=1}^{\ell}k_{i}(c_{i}-1)+1}\right)^{\frac{7}{2}}=\left(1+\frac{\sum_{i=1}^{\ell}k_{i}(c_{i}-1)}{n+1-\sum_{i=1}^{\ell}k_{i}(c_{i}-1)}\right)^{\frac{7}{2}}=1+O\left(\frac{\ell\tilde{k}\tilde{c}}{n}\right).
\]
Put together, we obtain 
\begin{align*}
 & \frac{C\left(\prod_{i=1}^{\ell}V_{c_{i}+1}^{k_{i}}\right)V_{n-\sum_{i=1}^{\ell}k_{i}(c_{i}-1)}}{n^{\sum_{i=1}^{\ell}k_{i}}V_{n}}\prod_{i=1}^{\ell}\left(\frac{b_{i}^{2}-a_{i}^{2}}{2}\right)^{k_{i}}\left(1+O_{\tilde{k},\tilde{b},\tilde{c},\ell}\left(\frac{1}{n}\right)\right)\\
 & =\left(\prod_{i=1}^{\ell}\left(\frac{b_{i}^{2}-a_{i}^{2}}{2}\frac{V_{c_{i}+1}}{c_{i}!}\left(\frac{x_{0}}{2\pi^{2}}\right)^{c_{i}-1}\right)^{k_{i}}\right)\left(1+O_{\tilde{k},\tilde{b},\tilde{c},\ell}\left(\frac{1}{n}\right)\right),
\end{align*}
as required.
\end{proof}
\begin{proof}[Proof of Theorem \ref{thm:main-thm-2}]
 The result now follows from Proposition \ref{prop:unnested-contribution-fixecusps}
and the method of factorial moments Proposition \ref{prop:factorial-moments}.
\end{proof}
\begin{proof}[Proof of Corollary \ref{cor:systole-type}]
 Notice by definition, that a surface satisfying 
\[
\left(N_{n,2,[0,x]}(X),\ldots,N_{n,k-1,[0,x]}(X),N_{n,k,[0,x]}(X)\right)=(0,\ldots,0,m)
\]
for some $m\geq1$ is contained in $\mathcal{A}_{k,x,n}$. By Theorem
\ref{thm:main-thm-2}, this sequence of random vectors converges in
distribution as $n\to\infty$ to a vector of independent Poisson random
variables with means 

\[
\lambda_{i,[0,x]}=\frac{x^{2}}{2}\frac{V_{i+1}}{i!}\left(\frac{x_{0}}{2\pi^{2}}\right)^{i-1}
\]
for $i=2,\ldots,k$ respectively. It thus follows that

\begin{align*}
\lim_{n\to\infty}\mathbb{P}_{n}\left(\mathcal{A}_{k,x,n}\right) & \geqslant\lim_{n\to\infty}\mathbb{P}_{n}\left(X:N_{n,2,[0,x]}(X)=0,\ldots,N_{n,k-1,[0,x]}(X)=0,N_{n,k,[0,x]}(X)\geq1\right)\\
 & =e^{-\sum_{i=2}^{k-1}\lambda_{i,[0,x]}}\left(1-e^{-\lambda_{k,[0,x]}}\right),
\end{align*}
as required.

\end{proof}

\section{\label{sec:Systole}Systole}

In this section, we prove Proposition \ref{prop:prop1.3}, providing
information about the systole of hyperbolic punctured spheres. Recall
that for an integer $c\geqslant2$ and real numbers $a,b\geqslant0$,
the random variable $N_{n,c,[a,b]}(X):\left(\mathcal{M}_{0,n},\mathbb{P}_{n}\right)\to\mathbb{N}$
counts the number of closed, primitive geodesics with lengths in $\left[\frac{a}{\sqrt{n}},\frac{b}{\sqrt{n}}\right]$
which separate off $c$ cusps. 
\begin{prop}[Proposition \ref{prop:prop1.3}]
\label{prop:systole-lb}There exists a constant $B>0$ such that
for any constants $0<\varepsilon<\frac{1}{2}$, $A>0$, any $c_{n}<An^{\varepsilon}$
and $n$ sufficiently large, 
\[
\mathbb{P}_{n}\left(\mathrm{sys}(X)>\frac{c_{n}}{\sqrt{n}}\right)\leq\max\left\{ Bc_{n}^{-2},\frac{B}{\sqrt{n}}\right\} .
\]
\end{prop}

\begin{proof}
The systole of a surface is always a simple closed geodesic, and so
if the systole has length greater than $\frac{c_{n}}{\sqrt{n}},$
then $N_{n,2,[0,c_{n}]}(X)=0$, thus we have
\[
\mathbb{P}_{n}\left(\mathrm{sys}(X)>\frac{c_{n}}{\sqrt{n}}\right)\leq\mathbb{P}_{n}\left(N_{n,2,[0,c_{n}]}(X)=0\right).
\]
By the second moment method
\begin{align*}
\mathbb{P}_{n}\left(N_{n,2,[0,c_{n}]}(X)=0\right) & \leq\frac{\mathbb{E}_{n}\left(\left(N_{n,2,[0,c_{n}]}(X)\right)^{2}\right)-\mathbb{E}_{n}\left(N_{n,2,[0,c_{n}]}(X)\right)^{2}}{\mathbb{E}_{n}\left(\left(N_{n,2,[0,c_{n}]}(X)\right)^{2}\right)}\\
 & =\frac{\mathbb{E}_{n}\left(\left(N_{n,2,[0,c_{n}]}(X)\right)_{2}\right)+\mathbb{E}_{n}\left(N_{n,2,[0,c_{n}]}(X)\right)-\mathbb{E}_{n}\left(N_{n,2,[0,c_{n}]}(X)\right)^{2}}{\mathbb{E}_{n}\left(\left(N_{n,2,[0,c_{n}]}(X)\right)_{2}\right)+\mathbb{E}_{n}\left(N_{n,2,[0,c_{n}]}(X)\right)}.
\end{align*}
Given any $A>0$ and $0<\varepsilon<\frac{1}{2}$, for $n$ sufficiently
large, $An^{\varepsilon-\frac{1}{2}}<2\mathrm{arcsinh}(1)$ and so
the second factorial moment of $N_{n,2,[0,c_{n}]}(X)$ has the natural
interpretation of counting the number of ordered multicurves of length
two consisting of distinct, non-intersecting simple closed geodesics
with lengths in $\left[0,\frac{c_{n}}{\sqrt{n}}\right]$ that separate
off 2 cusps. Using Theorem \ref{thm:MIF} we compute for any $t=o_{n\to\infty}\left(n^{\frac{1}{2}}\right)$,
\begin{align*}
\mathbb{E}_{n}\left(N_{n,2,[0,t]}(X)\right) & =\frac{1}{V_{0,n}}{n \choose 2}\int_{0}^{\frac{t}{\sqrt{n}}}xV_{0,3}(0,0,x)V_{0,n-1}(\mathbf{0}_{n-2},x)\mathrm{d}x\\
 & =\frac{V_{0,n-1}}{V_{0,n}}{n \choose 2}\int_{0}^{\frac{t}{\sqrt{n}}}4\frac{\sinh^{2}\left(\frac{x}{2}\right)}{x}\left(1+O\left(x^{2}\right)\right)\mathrm{d}x\\
 & =\frac{V_{0,n-1}}{V_{0,n}}{n \choose 2}\frac{t^{2}}{2n}\left(1+O\left(\frac{t^{2}}{n}\right)\right).
\end{align*}
By Theorem \ref{thm:ZografManin}, we obtain 
\begin{align*}
\frac{V_{0,n-1}}{V_{0,n}}{n \choose 2} & =\frac{1}{2}\frac{x_{0}}{2\pi^{2}}\frac{(n-1)!(n+1)^{\frac{7}{2}}}{n^{\frac{7}{2}}(n-2)!}\left(1+O\left(\frac{1}{n}\right)\right)\\
 & =\frac{1}{2}\frac{x_{0}}{2\pi^{2}}n\left(1+O\left(\frac{1}{n}\right)\right).
\end{align*}
Thus,
\[
\mathbb{E}_{n}\left(N_{n,2,[0,t]}(X)\right)=\frac{t^{2}x_{0}}{8\pi^{2}}\left(1+O\left(\frac{t^{2}}{n}\right)\right).
\]
Similarly, we compute 
\begin{align*}
\mathbb{E}_{n}\left(\left(N_{n,2,[0,t]}(X)\right)_{2}\right) & =\frac{1}{V_{0,n}}{n \choose 2,2}\int_{0}^{\frac{t}{\sqrt{n}}}\int_{0}^{\frac{t}{\sqrt{n}}}xyV_{0,3}(0,0,x)V_{0,3}(0,0,y)V_{0,n-2}(\mathbf{0}_{n-4},x,y)\mathrm{d}x\mathrm{d}y\\
 & =\frac{V_{0,n-2}}{V_{0,n}}{n \choose 2,2}\left(\int_{0}^{\frac{t}{\sqrt{n}}}4\frac{\sinh^{2}\left(\frac{x}{2}\right)}{x}\mathrm{d}x\right)^{2}\left(1+O\left(\frac{t^{2}}{n}\right)\right)\\
 & =\frac{V_{0,n-2}}{V_{0,n}}{n \choose 2,2}\frac{t^{4}}{4n^{2}}\left(1+O\left(\frac{t^{2}}{n}\right)\right).
\end{align*}
Again using Theorem \ref{thm:ZografManin}, we can compute 
\[
\mathbb{E}_{n}\left(\left(N_{n,2,[0,t]}(X)\right)_{2}\right)=\frac{t^{4}x_{0}^{2}}{64\pi^{4}}\left(1+O\left(\frac{t^{2}}{n}\right)\right).
\]
Using $t=c_{n}$ in these estimates, we find that there is a uniform
constant $B>0$ such that 
\begin{equation}
\mathbb{P}_{n}\left(N_{n,2,[0,c_{n}]}(X)=0\right)\leq\frac{1+\frac{c_{n}^{2}x_{0}}{8\pi^{2}}O\left(\frac{c_{n}^{2}}{n}\right)+O\left(\frac{c_{n}^{2}}{n}\right)}{\left(\frac{c_{n}^{2}x_{0}}{8\pi^{2}}+1\right)\left(1+O\left(\frac{c_{n}^{2}}{n}\right)\right)}\leq\frac{B}{c_{n}^{2}}+B\frac{c_{n}^{2}}{n}.\label{eq:P-num=00003D0}
\end{equation}
If $c_{n}\leqslant\sqrt{n}$ then (\ref{eq:P-num=00003D0}) gives
the claim. Otherwise if $c_{n}>\sqrt{n}$, we use that 
\[
\mathbb{P}_{n}\left(N_{n,2,[0,c_{n}]}(X)=0\right)\leqslant\mathbb{P}_{n}\left(N_{n,2,[0,\sqrt{n}]}(X)=0\right)
\]
and (\ref{eq:P-num=00003D0}) to finish the proof.
\end{proof}
\begin{rem}
The previous theorem also remains true for the moduli space $\mathcal{M}_{g,n}$
where $g$ is a fixed constant. This follows from large cusp volume
asymptotics of $\mathcal{M}_{g,n}$ from Manin and Zograf \cite[Theorem 6.1]{Ma.Zo00}
when $g$ is fixed and non-zero.
\end{rem}

\section{Small eigenvalues\label{sec:Spectrum}}

In this section we prove Theorem \ref{thm:small-eigenvalues}. First
we want to establish a geometric criterion for the existence of multiple
small eigenvalues. We start with the following mini-max principle
which can for example be found in \cite{Sm12}.
\begin{lem}[Mini-max]
\label{lem:Minimax}Let $A$ be a non-negative self-adjoint operator
on a Hilbert space $H$ with domain $\mathcal{D}\left(A\right)$.
Let $\lambda_{1}\leqslant\dots\leqslant\lambda_{k}$ denote the eigenvalues
of $A$ below the essential spectrum $\sigma_{ess}\left(A\right)$.
Then for $1\leqslant j\leqslant k,$
\[
\lambda_{j}=\min_{\psi_{1},\dots,\psi_{j}}\max\left\{ \frac{\langle\psi,A\psi\rangle}{\|\psi\|^{2}}\mid\psi\in\textup{span}\left(\psi_{1},\dots,\psi_{j}\right)\right\} ,
\]
where the minimum is taken over linearly independent $\psi_{1},\dots,\psi_{j}\in\mathcal{D}\left(A\right)$.
If $A$ only has $l\geqslant0$ eigenvalues below the essential spectrum
then for any integer $s\geqslant1$,
\[
\inf\sigma_{ess}\left(A\right)=\inf_{\psi_{1},\dots,\psi_{l+s}}\sup\left\{ \frac{\langle\psi,A\psi\rangle}{\|\psi\|^{2}}\mid\psi\in\textup{span}\left(\psi_{1},\dots,\psi_{l+s}\right)\right\} ,
\]
where again the infimum is taken over linearly independent $\psi_{1},\dots,\psi_{\ell+s}\in\mathcal{D}\left(A\right)$.
\end{lem}

Using Lemma \ref{lem:Minimax} we prove the following result which
gives us a criterion for the existence of small eigenvalues.
\begin{lem}
\label{lem:Mini-Max inequality}Let $X\in\mathcal{M}_{g,n}$ and assume
that there exists $f_{0},\dots,f_{k}\in C_{c}^{\infty}\left(X\right)$
with $\|f_{i}\|_{L^{2}}=1$ for $0\leqslant i\leqslant k$ and $\text{Supp}\left(f_{i}\right)\cap\text{Supp}\left(f_{j}\right)=\emptyset$
for $i\neq j$, such that 
\[
\max_{0\leqslant j\leqslant k}\int_{X}\|\textup{grad}f_{j}\left(z\right)\|^{2}\mathrm{d}\mu\left(z\right)<\kappa<\frac{1}{4}.
\]
Then $0<\lambda_{1}\left(X\right)\leqslant\dots\leqslant\lambda_{k}\left(X\right)$
exist and satisfy $\lambda_{k}\left(X\right)<\kappa$.
\end{lem}

\begin{proof}
Let $f_{0},\dots,f_{k}$ satisfy the assumptions of the Lemma. Since
$C_{c}^{\infty}\left(X\right)\subset\mathcal{D}\left(\Delta\right)$,
we see that 
\[
\min_{\psi_{1},\dots,\psi_{k+1}}\max\left\{ \frac{\langle\psi,\Delta\psi\rangle}{\|\psi\|^{2}}\mid\psi\in\text{span}\left(\psi_{1},\dots,\psi_{k+1}\right)\right\} \leqslant\max\left\{ \langle\psi,\Delta\psi\rangle\mid\psi\in\text{span}\left(f_{0},\dots,f_{k}\right),\|\psi\|_{L^{2}}=1\right\} .
\]
Let $\psi=\sum_{i=0}^{k}a_{i}f_{i}$ with $a_{i}\in\mathbb{C}$, then
$\|\psi\|_{L^{2}}=1$ implies that 
\[
1=\left\langle \sum_{i=0}^{k}a_{i}f_{i},\sum_{i=0}^{k}a_{i}f_{i}\right\rangle =\sum_{i=0}^{k}\left|a_{i}\right|^{2}\left\langle f_{i},f_{i}\right\rangle =\sum_{i=0}^{k}\left|a_{i}\right|^{2},
\]
and

\begin{align*}
\left\langle \psi,\Delta\psi\right\rangle  & =\left\langle \sum_{i=0}^{k}a_{i}f_{i},\Delta\sum_{i=0}^{k}a_{i}f_{i}\right\rangle =\sum_{i=0}^{k}\left|a_{i}\right|^{2}\left\langle f_{i},\Delta f_{i}\right\rangle <\kappa<\frac{1}{4},
\end{align*}
where we used that $\text{Supp}\left(f_{i}\right)\cap\text{Supp}\left(f_{j}\right)=\emptyset$
for $i\neq j$. We deduce that 
\[
\max\left\{ \langle\psi,A\psi\rangle\mid\psi\in\text{span}\left(f_{0},\dots,f_{k}\right),\|\psi\|_{L^{2}}=1\right\} <\kappa<\frac{1}{4}.
\]
 Assume that $X$ does not have $k+1$ eigenvalues below $\frac{1}{4}$,
then $\inf\sigma_{ess}\left(X\right)<\kappa<\frac{1}{4}$ by Lemma
\ref{lem:Minimax}, giving a contradiction. Then $\lambda_{1}\left(X\right)\leqslant\dots\leqslant\lambda_{k}\left(X\right)$
exist and 
\[
\lambda_{k}\left(X\right)<\kappa,
\]
by Lemma \ref{lem:Minimax}.
\end{proof}
\begin{lem}
\label{lem:geodesics_to_eigenvalues}Let $\ep<\frac{1}{4}$ and assume
that $X\in\mathcal{M}_{g,n}$ has $k+1$ geodesics of length $\leqslant\frac{\ep}{6}$,
each of which separates off two cusps. Then $\lambda_{k}\left(X\right)$
exists and satisfies 
\[
\lambda_{k}\left(X\right)\leqslant\ep.
\]
\end{lem}

\begin{proof}
We adapt the proof of \cite[Theorem 8.1.3]{Bu2010}. First label the
geodesics from $\gamma_{0},\dots,\gamma_{k}$. Let $Y_{j}$ be the
subsurface with two cusps with geodesic boundary $\gamma_{j}$. Denote
the two cusps bounded by $\gamma_{j}$ by $C_{j,1},C_{j,2}$ and their
unique length $\frac{\ep}{6}$ horocycles by $\beta_{j,1},\beta_{j,2}$.
Let $\tilde{Y}_{j}\subset Y_{j}$ denote the compact region of $Y_{j}$
bounded by $\beta_{j,1},\beta_{j,2}$. We then define the function
\[
g_{k}\left(z\right)\eqdf\begin{cases}
\text{dist}\left(z,\gamma_{k}\right) & \text{if \ensuremath{z\in\tilde{Y}_{k}} and \ensuremath{\text{dist}\left(z,\gamma_{k}\right)\leqslant1,} }\\
\text{dist}\left(z,\beta_{k,1}\right) & \text{if \ensuremath{z\in\tilde{Y}_{k}} and \ensuremath{\text{dist}\left(z,\beta_{k,1}\right)\leqslant1,} }\\
\text{dist}\left(z,\beta_{k,2}\right) & \text{if \ensuremath{z\in\tilde{Y}_{k}} and \ensuremath{\text{dist}\left(z,\beta_{k,2}\right)\leqslant1,} }\\
0 & \text{if }\ensuremath{z\notin\tilde{Y}_{k}},\\
1 & \text{otherwise, }
\end{cases}
\]
which is well-defined by the choice of $\varepsilon$. Then writing
\[
Z_{k}\eqdf\left\{ z\in\tilde{Y}_{k}\mid\text{dist}\left(z,\gamma_{k}\right)\leqslant1\right\} \cup\left\{ z\in\tilde{Y}_{k}\mid\text{dist}\left(z,\beta_{k,1}\right)\leqslant1\right\} \cup\left\{ z\in\tilde{Y}_{k}\mid\text{dist}\left(z,\beta_{k,1}\right)\leqslant1\right\} ,
\]
we have that $\|\text{grad}g_{k}\left(z\right)\|^{2}=1$ for $z\in Z_{k}$,
$\|\text{grad}g_{k}\left(z\right)\|^{2}=0$ elsewhere and one can
calculate that
\[
\text{Vol}\left(Z_{k}\right)\leq\frac{5}{6}\ep\sinh1.
\]
Note that each $g_{k}$ is compactly supported by the definition of
$\tilde{Y}_{k}$. By $L^{2}$ normalizing and smoothly approximating,
we can find functions $f_{0},\dots,f_{k}\in C_{c}^{\infty}\left(X\right)$
with $\|f_{i}\|_{L^{2}}=1$ for $0\leqslant i\leqslant k$ and $\text{Supp}\left(f_{i}\right)\cap\text{Supp}\left(f_{j}\right)=\emptyset$
for $i\neq j$, such that 
\[
\max_{0\leqslant j\leqslant k}\int_{X}\|\text{grad}f_{j}\left(z\right)\|^{2}\mathrm{d}\mu\left(z\right)<\varepsilon<\frac{1}{4},
\]
and the conclusion follows from Lemma \ref{lem:Mini-Max inequality}. 
\end{proof}
We conclude with the proof of Theorem \ref{thm:small-eigenvalues}.
\begin{thm}[Theorem \ref{thm:small-eigenvalues}]
\label{Variable n asymptotic}There is a constant $C>0$ such that
for any function $k:\mathbb{N}\to\mathbb{N}$ with $k=o\left(n\right)$
and $k\to\infty$ as $n\to\infty$, 

\[
\mathbb{P}_{n}\left(\lambda_{k}(X)<C\sqrt{\frac{k}{n}}\right)\to1,
\]
as $n\to\infty$. In particular, for any $\ep>0$, $\mathbb{P}_{n}\left[\lambda_{k}\left(X\right)<\ep\right]\to1$
as $n\to\infty$.
\end{thm}

\begin{proof}
For $X\in\mathcal{M}_{n}$, let $\tilde{N}_{n,2,[0,l]}(X)$ be the
function that counts the number of primitive closed geodesics on $X$
which separate off $2$ cusps with lengths in the window $\left[0,l\right]$,
note that here there is no rescaling of the window. We claim that
there exists a constant $C'>0$ such that for any function $L:\mathbb{N}\to(0,\infty)$
with $L(n)\to\infty$ as $n\to\infty$ and $L=o\left(\sqrt{n}\right)$,
the probability that $X\in\mathcal{M}_{0,n}$ satisfies 
\begin{equation}
\tilde{N}_{n,2,\left[0,\frac{1}{6L}\right]}\left(X\right)\geqslant C'\frac{n}{L^{2}},\label{eq:number-short-curves}
\end{equation}
tends to $1$ as $n\to\infty$. We now show that Theorem \ref{Variable n asymptotic}
follows from this claim. Let $k=o\left(n\right)$ and pick $L$ with
$L\to\infty$ and $L\leqslant\sqrt{\frac{C'n}{k+1}}$. It follows
from the above claim that with probability tending to $1$, there
are at least $k+1$ curves with lengths $\leqslant\frac{1}{6L}$.
Then by Lemma \ref{lem:geodesics_to_eigenvalues}, $\lambda_{k}<\frac{1}{L}$.The
remainder of the proof is dedicated to showing (\ref{eq:number-short-curves}).

Consider the sequence of random variables
\[
Y_{n}\eqdf\frac{L^{2}\tilde{N}_{n,2,\left[0,\frac{1}{L}\right]}\left(X\right)}{n}.
\]
First we want to show that $\text{Var}\left(Y_{n}\right)\to0$ as
$n\to\infty$. By nearly identical calculations to the proof of Proposition
\ref{prop:systole-lb}, we see that 
\begin{align}
\mathbb{E}_{n}\left[Y_{n}\right] & ={n \choose 2}\frac{L^{2}}{nV_{n}}\int_{0}^{\frac{1}{L}}lV_{n-1}(l)dl=\frac{x_{0}}{8\pi^{2}}+o(1).\label{eq:exp-calc}
\end{align}
Calculating the second moment,
\begin{align*}
\mathbb{E}\left[\left(\frac{L^{2}\tilde{N}_{n,2,\left[0,\frac{1}{L}\right]}\left(X\right)}{n}\right)^{2}\right] & =\mathbb{E}\left[\left(\frac{L^{2}}{n}\sum_{\substack{\gamma\in\mathcal{P}(X)\\
\gamma\ \text{separates off exactly \ensuremath{2} cusps}
}
}\ind_{\frac{1}{L}}\left(l_{\gamma}\left(X\right)\right)\right)^{2}\right]\\
 & =\frac{L^{4}}{n^{2}}\mathbb{E}\left[\sum_{\substack{\gamma\in\mathcal{P}(X)\\
\gamma\ \text{separates off exactly \ensuremath{2} cusps}
}
}\ind_{\frac{1}{L}}\left(l_{\gamma}\left(X\right)\right)\right]\\
 & +\frac{L^{4}}{n^{2}}\mathbb{E}\left[\sum_{\substack{(\gamma_{1},\gamma_{2})\in\mathcal{P}\left(X\right)\times\mathcal{P}\left(X\right)\\
\gamma_{1}\neq\gamma_{2}\\
\gamma_{1},\gamma_{2}\ \text{separate off exactly \ensuremath{2} cusps}
}
}\ind_{\frac{1}{L}}\left(l_{\gamma_{1}}\left(X\right),l_{\gamma_{2}}\left(X\right)\right)\right].
\end{align*}
Then 
\[
\frac{L^{4}}{n^{2}}\mathbb{E}\left[\sum_{\substack{\gamma\in\mathcal{P}(X)\\
\gamma\ \text{separates off exactly \ensuremath{2} cusps}
}
}\ind_{\frac{1}{L}}\left(l_{\gamma}\left(X\right)\right)\right]=\frac{L^{2}}{n}\mathbb{E}\left[Y_{n}\right]=o\left(1\right).
\]
Again by similar calculations to the proof of Proposition \ref{prop:systole-lb}
and the fact that the pairs of curves $\gamma_{1}$ and $\gamma_{2}$
are disjoint when they have length at most $\frac{1}{L}\to0$,
\begin{align*}
\frac{L^{4}}{n^{2}}\mathbb{E}\left[\sum_{\substack{(\gamma_{1},\gamma_{2})\in\mathcal{P}\left(X\right)\times\mathcal{P}\left(X\right)\\
\gamma_{1}\neq\gamma_{2}\\
\gamma_{1},\gamma_{2}\ \text{separate off exactly \ensuremath{2} cusps}
}
}\ind_{\frac{1}{L}}\left(l_{\gamma_{1}}\left(X\right),l_{\gamma_{2}}\left(X\right)\right)\right] & ={n \choose 2,2}\frac{L^{4}}{n^{2}V_{n}}\int_{0}^{\frac{1}{L}}\int_{0}^{\frac{1}{L}}lV_{n-2}(l_{1},l_{2})dl_{1}dl_{2}\\
 & =\left(\frac{x_{0}}{8\pi^{2}}\right)^{2}+o\left(1\right).
\end{align*}
We conclude that as $n\to\infty$.
\begin{equation}
\text{Var}\left[Y_{n}\right]=\mathbb{E}\left[Y_{n}^{2}\right]-\mathbb{E}_{n}\left[Y_{n}\right]^{2}=o\left(1\right).\label{eq:Variance}
\end{equation}
By Chebyshev's inequality, we have 
\[
\mathbb{P}\left[\left|Y_{n}-\mathbb{E}\left[Y_{n}\right]\right|\geqslant\frac{x_{0}}{16\pi^{2}}\right]\leqslant\frac{256\pi^{4}}{x_{0}^{2}}\text{Var}\left[Y_{n}\right]\to0,
\]
as $n\to\infty$. In particular, with probability at least $1-\frac{256\pi^{4}}{x_{0}^{2}}\text{Var}\left[Y_{n}\right]$,
one has $Y_{n}>\mathbb{E}\left[Y_{n}\right]-\frac{x_{0}}{16\pi^{2}}$.
Moreover, by (\ref{eq:exp-calc}), there exists $A>0$ such that for
all $n\geq A$, one has $\mathbb{E}\left[Y_{n}\right]>\frac{3}{2}\frac{x_{0}}{16\pi^{2}}$.
Put together, this shows that $Y_{n}\geq\frac{x_{0}}{32\pi^{2}}$
for any $n\geq A$, with probability at least $1-\frac{256\pi^{4}}{x_{0}^{2}}\text{Var}\left[Y_{n}\right]$.
It follows that 
\[
\tilde{N}_{n,2,\left[0,\frac{1}{6L}\right]}\left(X\right)\geqslant\frac{x_{0}}{32\pi^{2}}\frac{n}{(6L)^{2}},
\]
with probability tending to $1$ as $n\to\infty$.
\end{proof}
\begin{rem}
\label{rem:g>0}The proof of Theorem \ref{thm:small-eigenvalues}
also works for surfaces with $\textit{fixed}$ genus $g>0$ with $n\to\infty$,
i.e. there is a constant $C>0$ such that for any function $k:\mathbb{N}\to\mathbb{N}$
with $k=o\left(n\right)$ and $k\to\infty$ as $n\to\infty$ and any
fixed $g$, $\mathbb{P}_{g,n}\left(\lambda_{k}(X)<C\sqrt{\frac{k}{n}}\right)\to1$
as $n\to\infty$. For this, one simply uses \cite[Theorem 6.1]{Ma.Zo00}
with $g\geqslant0$ fixed in the calculations leading to (\ref{eq:Variance}). 
\end{rem}

\begin{rem}
The proof Theorem \ref{thm:small-eigenvalues} shows that for $L=L\left(n\right)$
with $L\to\infty$ as $n\to\infty$ and $L=o\left(\sqrt{n}\right)$,
the random variable $\frac{L^{2}\tilde{N}_{n,2,\left[0,\frac{1}{L}\right]}(X)}{n}$
converges in distribution to the constant distribution $\frac{x_{0}}{8\pi^{2}}$.
\end{rem}

\section*{Acknowledgments}

JT was supported by funding from the European Research Council (ERC)
under the European Union’s Horizon 2020 research and innovation programme
(grant agreement No 949143). We thank Michael Magee for discussions
about this work, in particular with regards to the computation of
the constant $\sum_{i=2}^{\infty}\frac{V_{i+1}}{i!}\left(\frac{x_{0}}{2\pi^{2}}\right)^{i-1}$.
We thank Yang Shen and Yunhui Wu for comments on an earlier version
of this work about the topological types of systolic curves which
inspired Theorem \ref{thm:main-thm-2}. We thank the anonymous referee
for their comments and corrections and an improvment of Theorem 5.4.

\bibliographystyle{amsalpha}
\bibliography{puncturedspheresbib}

\noindent Will Hide, \\
Department of Mathematical Sciences,\\
Durham University, \\
Lower Mountjoy, DH1 3LE Durham,\\
United Kingdom

\noindent \texttt{william.hide@durham.ac.uk}~\\
\texttt{}~\\

\noindent Joe Thomas, \\
Department of Mathematical Sciences,\\
Durham University, \\
Lower Mountjoy, DH1 3LE Durham,\\
United Kingdom

\noindent \texttt{joe.thomas@durham.ac.uk}
\end{document}